\documentclass[12pt]{amsart}
\usepackage{amssymb, a4wide, mathdots,url,hyperref, shuffle}
\usepackage{bbm}
\usepackage{tikz}

\renewcommand{\subset}{\subseteq}

\newtheorem{theorem}{Theorem}[section]
\newtheorem{lemma}[theorem]{Lemma}
\newtheorem{proposition}[theorem]{Proposition}
\newtheorem{remark}[theorem]{Remark}

\newtheorem{definition}[theorem]{Definition}
\newtheorem{corollary}[theorem]{Corollary}
\newtheorem{example}[theorem]{Example}
\newtheorem{observation}[theorem]{Observation}

\newtheorem{claim}{Claim}

\date{\today}

\newcommand{\gotM}{\mathfrak{m}}
\newcommand{\gotN}{\mathfrak{n}}

\DeclareMathOperator{\irr}{Irr}
\DeclareMathOperator{\seg}{Seg}
\DeclareMathOperator{\ch}{ch}
\DeclareMathOperator{\supp}{supp}

\begin{document}

\title[Quantum invariants for decomposition problems]{Quantum invariants for decomposition problems in type $A$ rings of representations}
\author{Maxim Gurevich}
\address{ Department of Mathematics, Technion - Israel Institute of Technology, 3200003, Haifa, Israel}
\email{maxg@technion.ac.il}

\begin{abstract}
  We prove a combinatorial rule for a complete decomposition, in terms of Langlands parameters, for representations of $p$-adic $GL_n$ that appear as parabolic induction from a large family (ladder representations). Our rule obviates the need for computation of Kazhdan-Lusztig polynomials in these cases, and settles a conjecture posed by Lapid.

  These results are transferrable into various type $A$ frameworks, such as the decomposition of convolution products of homogeneous KLR-algebra modules, or tensor products of snake modules over quantum affine algebras.

  The method of proof applies a quantization of the problem into a question on Lusztig's dual canonical basis and its embedding into a quantum shuffle algebra, while computing numeric invariants which are new to the $p$-adic setting.

\end{abstract}

\maketitle

\section{Introduction}

Let $\mathcal{R}_n$ be the Grothendieck group associated with the category of complex-valued smooth finite-length representations of the group $GL_n(F)$, where $F$ is a $p$-adic field. The larger group $\mathcal{R} = \oplus_{n\geq 0} \mathcal{R}_n$ is equipped with a structure of a commutative ring, coming from the operation of parabolic induction. The multiplicative behavior of irreducible representations $\irr = \cup_{n\geq 0} \irr(GL_n(F))$ as elements of $\mathcal{R}$ remains largely a mystery.

Given two representations $\pi_1,\pi_2\in \irr$, the decomposition of $\pi_1\times\pi_2\in \mathcal{R}$ into irreducible factors can be given in terms of values of Kazhdan-Lusztig polynomials of a corresponding symmetric group (see e.g. \cite{Hender}). Yet, a recent line of research \cite{LM2,LM3, me-decomp} joins some classical results \cite{bern} in an attempt to produce more transparent descriptions of such decompositions, while significantly reducing the computational complexity involved.

In a previous paper \cite[Conjecture 7.1]{me-decomp} we have given a conjectural effective answer for the decomposition of $\pi_1\times \pi_2$ for the case when $\pi_1,\pi_2$ belong to a class we call ladder representations. In particular, we have generalized a key conjecture by Lapid that for a specific choice of representations, $\pi_1\times \pi_2$ would be of length $C_n$ ($n$-th Catalan number).

Our main result here (Theorem \ref{thm-main}) is the proof of that decomposition.

While we refrain from fine algorithm complexity analysis here, let us stress that our new algorithm gives an answer (for whether a given irreducible representation appears as a factor in a given product) in a time that is bound by a polynomial in the length of the input (number of segments which define $\pi_1,\pi_2$). Thus far, the general algorithm for computing such decompositions would involve computing summation formulas containing a list of values of Kazhdan-Lusztig polynomials, which are obtained by a slow recursive process. This naive solution would also lack a combinatorial description of the resulting subquotients, as opposed to our approach in this paper.


We recall that the decomposition problems at hand are inherent in various representation-theoretic settings defined by type $A$ data. Thus, while our discussion is set in the domain of $p$-adic groups, the results are easily transferable to other rings of interest.

First, the problem can be stated as a decomposition of an induction product of two irreducible finite-dimensional modules of (extended) affine Hecke algebras attached to the $GL_n$ root data. In the Hecke algebra setting, ladder representations may be translated to the notions of calibrated modules in the sense of \cite{ram} or of $\bullet$-unitary modules as explained in \cite{dandan}.

Second, by the quantum affine Schur-Weyl duality of \cite{cp-duality}, our problem is equivalent to the decomposition of tensor products of simple modules of the quantum affine algebra $U_q(\hat{\frak{sl}}_N)$. The analog of ladder representations in that setting was studied in \cite{naztar} and received the name \textit{snake modules} in the follow-up work \cite{MukhYoung}. The decomposition of certain cases of our problem was shown in \cite{nak03} to comply with so-called $T$-systems, which are of interest in mathematical physics.

Observing all above mentioned settings, our results can be viewed as a generalization of previous decomposition descriptions of Tadi\'c \cite{tadic-speh}, Leclerc \cite{lecl} and Ram \cite{ram}.

The crucial involvement of the Robinson-Schensted correspondence in the proof of the main result (Section \ref{sec-rs}) has lead Lapid and the author to introduce the more general notion of \textit{RSK-standard modules} in \cite{gur-lap}.

\subsection{Main result and examples}
The collection of (isomorphism classes of) irreducible representations $\irr$ is known to be parameterized, through the Zelevinski/Langlands classification, by \textit{multisegments}, that is, multi-sets of formal objects called segments.

Through this description, we explore a combinatorial point of view on the problem. Given $\pi_1,\pi_2\in \irr$, a pair which is parameterized by multisegments $\gotM_1,\gotM_2$, respectively, we would like to obtain a combinatorial algorithm that produces all multisegments that parameterize the irreducible subquotients of $\pi_1\times\pi_2$.

While in general segments encode in their data some supercuspidal representations, the core of our problem lies in the case when, in a suitable sense (see Section \ref{sect:lines}), this part of the data is fixed and irrelevant. Thus, for exposition purposes we will assume $\gotM_1, \gotM_2$ are multi-sets of formal segments of the form $[a,b]$, for integers $a\leq b$. We will write them additively (proper notation will be explained in the body of the paper).

For example,
\[
\gotM_1 = [10,12] + [12,14] + [13,16],\quad\gotM_2 = [11,13]+[14,15]\;.
\]

Let us also write $\lambda_1\leq \ldots\leq \lambda_n$ and $\mu_1\leq \ldots\leq \mu_n$ for the tuples of integers that comprise the begin and end points of the segments in $\gotM_1+\gotM_2$. In other words,
\[
\gotM_1+ \gotM_2 = \sum_{i=1}^n [\lambda_i , \mu_{\omega(i)}]\;,
\]
for a permutation $\omega\in S_n$.

In the previous example,
\[
(\lambda_1,\ldots, \lambda_5) = (10,11,12,13,14),\; (\mu_1,\ldots, \mu_5) = (12,13,14,15,16),\,\omega = (12354)\;.
\]

\begin{center}
\begin{tikzpicture}
\draw[step=1cm,gray,very thin] (9.5,0) grid (16.5,5.5);
\draw[ultra thick,->] (9.5,0) -- (16.5,0);
\foreach \x in {10,11,12,13,14,15,16}
   \draw[thick] (\x cm,0.2cm) -- (\x cm,-0.2cm) node[anchor=north] {$\x$};

\draw[red, very thick] (10,1) -- (12,1);
\fill[red] (10,1) circle (0.1cm) node[anchor=south east] {$\lambda_1$};
\fill[red] (12,1) circle (0.1cm) node[anchor=south west] {$\mu_1$};

\draw[blue, very thick] (11,2) -- (13,2);
\fill[blue] (11,2) circle (0.1cm) node[anchor=south east] {$\lambda_2$};
\fill[blue] (13,2) circle (0.1cm) node[anchor=south west] {$\mu_2$};

\draw[red, very thick] (12,3) -- (14,3);
\fill[red] (12,3) circle (0.1cm) node[anchor=south east] {$\lambda_3$};
\fill[red] (14,3) circle (0.1cm) node[anchor=south west] {$\mu_3$};

\draw[red, very thick] (13,4) -- (16,4);
\fill[red] (13,4) circle (0.1cm) node[anchor=south east] {$\lambda_4$};
\fill[red] (16,4) circle (0.1cm) node[anchor=south west] {$\mu_5$};

\draw[blue, very thick] (14,5) -- (15,5);
\fill[blue] (14,5) circle (0.1cm) node[anchor=south east] {$\lambda_5$};
\fill[blue] (15,5) circle (0.1cm) node[anchor=south west] {$\mu_4$};

\draw[red] (15.5,1.35) node {\huge $\gotM_1$};

\draw[blue] (10.5,4.35) node {\huge $\gotM_2$};

\end{tikzpicture}
\end{center}

This allows us to parameterize the multisegments which can potentially correspond to irreducible subquotients of the product $\pi_1\times \pi_2$, by permutations in the group $S_n$ (see \cite[Proposition 3.4]{me-decomp} or Section \ref{sec-main}).

Namely, for $x\in S_n$, let us write $\Pi(x)\in \irr$ for the representation that corresponds to the multisegment
\[
\gotM(x) =  \sum_{i=1}^n [\lambda_i , \mu_{x(i)}]\;,
\]
if well-defined\footnote{That is, when $\lambda_i-1 \leq \mu_{x(i)}$, for all $1\leq i\leq n$. Segments of the form $[a,a-1]$ are considered as an empty segment.}.


In \cite[Theorem 1.2]{me-decomp}, it was shown that products of two ladder representations have multiplicity-free subquotients. In other words, assuming $\pi_1,\pi_2$ are ladder representations, we have
\[
[\pi_1\times\pi_2] = \sum_{x\in S(\pi_1,\pi_2)} [\Pi(x)] \;\in \mathcal{R}\;,
\]
for a certain subset $S(\pi_1,\pi_2)\subset S_n$.

Theorem \ref{thm-main} gives an effective description of the set $S(\pi_1,\pi_2)$.

A crucial ingredient in this description is the \textit{indicator representation}. Given any $\sigma\in \irr(GL_m(F))$, we define its indicator representation $\sigma_\otimes$ to be a certain irreducible representation of a standard Levi subgroup of $GL_m(F)$ (See Definition \ref{defi:indi} for the precise notion).

\begin{theorem}[Approximate version of Theorem \ref{thm-main}]
For ladder representations $\pi_1,\pi_2\in \irr$, the set $S(\pi_1,\pi_2)$ consists of those permutations $x\in S_n$ which avoid a $321$ pattern and for which the indicator representation $\Pi(x)_\otimes$ appears in the Jacquet module of $\pi_1\times \pi_2$.
\end{theorem}

Given a permutation $x\in S_n$, it is a simple computational task to determine whether $\Pi(x)_\otimes$ appears in the Jacquet module of $\pi_1\times \pi_2$. This is due to Mackey theory (i.e. Geometric Lemma of Bernstein-Zelevinski) and the simple description of Jacquet modules of ladder representation achieved in \cite{LapidKret}.

In our chosen example, $\pi_1, \pi_2$ are in fact in the ladder class, meaning that both begin and end points of each of $\gotM_1, \gotM_2$ are strictly ascending. Let us demonstrate the combinatorial algorithm resulting from Theorem \ref{thm-main} through this example.

Take $x_1 = (34152)\in S_n$. The multisegment $\gotM(x_1)$ is shown in the following picture.

\begin{center}
\begin{tikzpicture}
\draw[step=1cm,gray,very thin] (9.5,0) grid (16.5,4.5);
\draw[ultra thick,->] (9.5,0) -- (16.5,0);
\foreach \x in {10,11,12,13,14,15,16}
   \draw[thick] (\x cm,0.2cm) -- (\x cm,-0.2cm) node[anchor=north] {$\x$};

\draw[brown, very thick] (10,1) -- (14,1);
\fill[brown] (10,1) circle (0.1cm) node[anchor=south east] {$\lambda_1$};
\fill[brown] (14,1) circle (0.1cm) node[anchor=south west] {$\mu_3$};

\draw[green, very thick] (11,2) -- (15,2);
\fill[green] (11,2) circle (0.1cm) node[anchor=south east] {$\lambda_2$};
\fill[green] (15,2) circle (0.1cm) node[anchor=south west] {$\mu_4$};

\fill[teal] (12,3) circle (0.1cm) node[anchor=south] {$\lambda_3=\mu_1$};

\draw[purple, very thick] (13,4) -- (16,4);
\fill[purple] (13,4) circle (0.1cm) node[anchor=south east] {$\lambda_4$};
\fill[purple] (16,4) circle (0.1cm) node[anchor=south west] {$\mu_5$};


\draw (14,5) node {$\left([\lambda_5,\mu_2] = \emptyset\right)$};


\end{tikzpicture}
\end{center}

Diagrammatically, the condition of $\Pi(x_1)_\otimes$ appearing in the Jacquet module of $\pi_1\times \pi_2$, means that the segments of $\gotM(x_1)$ can ``tile" $\gotM_1$ and $\gotM_2$, in such manner that each consecutive segment would cover a ``sub-ladder" of both $\gotM_1$ and $\gotM_2$ (see Example \ref{exampl} for a more precise statement).

\begin{center}
\begin{tikzpicture}
\draw[step=1cm,gray,very thin] (9.5,0) grid (16.5,3.5);
\draw[ultra thick,->] (9.5,0) -- (16.5,0);
\foreach \x in {10,11,12,13,14,15,16}
   \draw[thick] (\x cm,0.2cm) -- (\x cm,-0.2cm) node[anchor=north] {$\x$};

\draw[brown, very thick] (10,1) -- (12,1);
\fill[red] (10,1)  node[anchor=south east] {$\lambda_1$};
\fill[red] (12,1)  node[anchor=south west] {$\mu_1$};

\fill[brown] (10,1) circle (0.1cm);
\fill[brown] (12,1) circle (0.1cm);

\draw[red,  thick] (12,2) -- (13,2);
\fill[red] (12,2) node[anchor=south east] {$\lambda_3$};
\fill[red] (14,2)  node[anchor=south west] {$\mu_3$};

\fill[red] (12,2) circle (0.05cm);
\fill[red] (13,2) circle (0.05cm);
\fill[brown] (14,2) circle (0.1cm);

\draw[red,  thick] (13,3) -- (16,3);
\fill[red] (13,3) circle (0.05cm) node[anchor=south east] {$\lambda_4$};
\fill[red] (16,3) circle (0.05cm)  node[anchor=south west] {$\mu_5$};

\end{tikzpicture}
\qquad
\begin{tikzpicture}
\draw[step=1cm,gray,very thin] (9.5,0) grid (16.5,3.5);
\draw[ultra thick,->] (9.5,0) -- (16.5,0);
\foreach \x in {10,11,12,13,14,15,16}
   \draw[thick] (\x cm,0.2cm) -- (\x cm,-0.2cm) node[anchor=north] {$\x$};

\draw[brown, very thick] (10,1) -- (12,1);
\fill[red] (10,1)  node[anchor=south east] {$\lambda_1$};
\fill[red] (12,1)  node[anchor=south west] {$\mu_1$};

\fill[brown] (10,1) circle (0.1cm);
\fill[brown] (12,1) circle (0.1cm);

\fill[red] (12,2)  circle (0.05cm) node[anchor=south east] {$\lambda_3$};
\fill[red] (14,2)  node[anchor=south west] {$\mu_3$};

\fill[green] (13,2) circle (0.1cm);
\fill[brown] (14,2) circle (0.1cm);

\draw[red,  thick] (13,3) -- (16,3);
\fill[red] (13,3) circle (0.05cm) node[anchor=south east] {$\lambda_4$};
\fill[red] (16,3) circle (0.05cm) node[anchor=south west] {$\mu_5$};

\end{tikzpicture}

\end{center}

\begin{center}
\begin{tikzpicture}
\draw[step=1cm,gray,very thin] (9.5,0) grid (16.5,3.5);
\draw[ultra thick,->] (9.5,0) -- (16.5,0);
\foreach \x in {10,11,12,13,14,15,16}
   \draw[thick] (\x cm,0.2cm) -- (\x cm,-0.2cm) node[anchor=north] {$\x$};

\draw[blue,  thick] (11,1) -- (12,1);
\fill[blue] (11,1) circle (0.05cm) node[anchor=south east] {$\lambda_2$};
\fill[blue] (13,1)  node[anchor=south west] {$\mu_2$};

\fill[blue] (12,1) circle (0.05cm);
\fill[brown] (13,1) circle (0.1cm);

\draw[blue, thick] (14,2) -- (15,2);
\fill[blue] (14,2) circle (0.05cm) node[anchor=south east] {$\lambda_5$};
\fill[blue] (15,2) circle (0.05cm) node[anchor=south west] {$\mu_4$};
\draw (13,-1) node[anchor=north] {Step 1};
\end{tikzpicture}
\qquad
\begin{tikzpicture}
\draw[step=1cm,gray,very thin] (9.5,0) grid (16.5,3.5);
\draw[ultra thick,->] (9.5,0) -- (16.5,0);
\foreach \x in {10,11,12,13,14,15,16}
   \draw[thick] (\x cm,0.2cm) -- (\x cm,-0.2cm) node[anchor=north] {$\x$};

\draw[green, very thick] (11,1) -- (12,1);
\fill[blue] (11,1)  node[anchor=south east] {$\lambda_2$};
\fill[blue] (13,1)  node[anchor=south west] {$\mu_2$};

\fill[brown] (13,1) circle (0.1cm);

\fill[green] (11,1) circle (0.1cm);
\fill[green] (12,1) circle (0.1cm);

\draw[green, very thick] (14,2) -- (15,2);
\fill[blue] (14,2)  node[anchor=south east] {$\lambda_5$};
\fill[blue] (15,2)  node[anchor=south west] {$\mu_4$};

\fill[green] (14,2) circle (0.1cm);
\fill[green] (15,2) circle (0.1cm);

\draw (13,-1) node[anchor=north] {Step 2};
\end{tikzpicture}
\end{center}

\begin{center}
\begin{tikzpicture}
\draw[step=1cm,gray,very thin] (9.5,0) grid (16.5,3.5);
\draw[ultra thick,->] (9.5,0) -- (16.5,0);
\foreach \x in {10,11,12,13,14,15,16}
   \draw[thick] (\x cm,0.2cm) -- (\x cm,-0.2cm) node[anchor=north] {$\x$};

\draw[brown, very thick] (10,1) -- (12,1);
\fill[red] (10,1)  node[anchor=south east] {$\lambda_1$};
\fill[red] (12,1)  node[anchor=south west] {$\mu_1$};

\fill[brown] (10,1) circle (0.1cm);
\fill[brown] (12,1) circle (0.1cm);

\fill[red] (12,2)  node[anchor=south east] {$\lambda_3$};
\fill[red] (14,2)  node[anchor=south west] {$\mu_3$};

\fill[teal] (12,2) circle (0.1cm);
\fill[green] (13,2) circle (0.1cm);
\fill[brown] (14,2) circle (0.1cm);

\draw[red,  thick] (13,3) -- (16,3);
\fill[red] (13,3) circle (0.05cm) node[anchor=south east] {$\lambda_4$};
\fill[red] (16,3) circle (0.05cm) node[anchor=south west] {$\mu_5$};

\end{tikzpicture}
\qquad
\begin{tikzpicture}
\draw[step=1cm,gray,very thin] (9.5,0) grid (16.5,3.5);
\draw[ultra thick,->] (9.5,0) -- (16.5,0);
\foreach \x in {10,11,12,13,14,15,16}
   \draw[thick] (\x cm,0.2cm) -- (\x cm,-0.2cm) node[anchor=north] {$\x$};

\draw[brown, very thick] (10,1) -- (12,1);
\fill[red] (10,1)  node[anchor=south east] {$\lambda_1$};
\fill[red] (12,1)  node[anchor=south west] {$\mu_1$};

\fill[brown] (10,1) circle (0.1cm);
\fill[brown] (12,1) circle (0.1cm);

\fill[red] (12,2)  node[anchor=south east] {$\lambda_3$};
\fill[red] (14,2)  node[anchor=south west] {$\mu_3$};

\fill[teal] (12,2) circle (0.1cm);
\fill[green] (13,2) circle (0.1cm);
\fill[brown] (14,2) circle (0.1cm);

\draw[purple,  very thick] (13,3) -- (16,3);
\fill[red] (13,3)  node[anchor=south east] {$\lambda_4$};
\fill[red] (16,3)  node[anchor=south west] {$\mu_5$};

\fill[purple] (13,3) circle (0.1cm);
\fill[purple] (16,3) circle (0.1cm);

\end{tikzpicture}

\end{center}

\begin{center}
\begin{tikzpicture}
\draw[step=1cm,gray,very thin] (9.5,0) grid (16.5,3.5);
\draw[ultra thick,->] (9.5,0) -- (16.5,0);
\foreach \x in {10,11,12,13,14,15,16}
   \draw[thick] (\x cm,0.2cm) -- (\x cm,-0.2cm) node[anchor=north] {$\x$};

\draw[green, very thick] (11,1) -- (12,1);
\fill[blue] (11,1)  node[anchor=south east] {$\lambda_2$};
\fill[blue] (13,1)  node[anchor=south west] {$\mu_2$};

\fill[brown] (13,1) circle (0.1cm);

\fill[green] (11,1) circle (0.1cm);
\fill[green] (12,1) circle (0.1cm);

\draw[green, very thick] (14,2) -- (15,2);
\fill[blue] (14,2)  node[anchor=south east] {$\lambda_5$};
\fill[blue] (15,2)  node[anchor=south west] {$\mu_4$};

\fill[green] (14,2) circle (0.1cm);
\fill[green] (15,2) circle (0.1cm);

\draw (13,-1) node[anchor=north] {Step 3};
\end{tikzpicture}
\qquad
\begin{tikzpicture}
\draw[step=1cm,gray,very thin] (9.5,0) grid (16.5,3.5);
\draw[ultra thick,->] (9.5,0) -- (16.5,0);
\foreach \x in {10,11,12,13,14,15,16}
   \draw[thick] (\x cm,0.2cm) -- (\x cm,-0.2cm) node[anchor=north] {$\x$};

\draw[green, very thick] (11,1) -- (12,1);
\fill[blue] (11,1)  node[anchor=south east] {$\lambda_2$};
\fill[blue] (13,1)  node[anchor=south west] {$\mu_2$};

\fill[brown] (13,1) circle (0.1cm);

\fill[green] (11,1) circle (0.1cm);
\fill[green] (12,1) circle (0.1cm);

\draw[green, very thick] (14,2) -- (15,2);
\fill[blue] (14,2)  node[anchor=south east] {$\lambda_5$};
\fill[blue] (15,2)  node[anchor=south west] {$\mu_4$};

\fill[green] (14,2) circle (0.1cm);
\fill[green] (15,2) circle (0.1cm);

\draw (13,-1) node[anchor=north] {Step 4};

\end{tikzpicture}
\end{center}

The pictures above exhibit that $x_1$ indeed satisfies the Jacquet module conditions. We see how in Step $j$ the multisegment $ \sum_{i=1}^j [\lambda_i , \mu_{x_1(i)}]$ ``tiles sub-ladders" of $\gotM_1$ and $\gotM_2$.

Since it also clearly avoids a $321$ pattern, we conclude that $x_1\in S(\pi_1,\pi_2)$.

Now, $x_2 = (35142)$ satisfies the Jacquet module condition as well, as exhibited below. Yet, since $x_2(2) > x_2(4) > x_2(5)$ holds, $x_2\not\in S(\pi_1,\pi_2)$.

\begin{minipage}{0.5\textwidth}
\begin{tikzpicture}
\draw[step=1cm,gray,very thin] (9.5,0) grid (16.5,4.5);
\draw[ultra thick,->] (9.5,0) -- (16.5,0);
\foreach \x in {10,11,12,13,14,15,16}
   \draw[thick] (\x cm,0.2cm) -- (\x cm,-0.2cm) node[anchor=north] {$\x$};

\draw[brown, very thick] (10,1) -- (14,1);
\fill[brown] (10,1) circle (0.1cm) node[anchor=south east] {$\lambda_1$};
\fill[brown] (14,1) circle (0.1cm) node[anchor=south west] {$\mu_3$};

\draw[green, very thick] (11,2) -- (16,2);
\fill[green] (11,2) circle (0.1cm) node[anchor=south east] {$\lambda_2$};
\fill[green] (16,2) circle (0.1cm) node[anchor=south west] {$\mu_5$};

\fill[teal] (12,3) circle (0.1cm) node[anchor=south] {$\lambda_3=\mu_1$};

\draw[purple, very thick] (13,4) -- (15,4);
\fill[purple] (13,4) circle (0.1cm) node[anchor=south east] {$\lambda_4$};
\fill[purple] (15,4) circle (0.1cm) node[anchor=south west] {$\mu_4$};


\draw (13,-1) node[anchor=north] {\Large $\gotM(x_2)$};


\end{tikzpicture}
\end{minipage}
\begin{minipage}{0.5\textwidth}
\begin{tikzpicture}
\draw[step=1cm,gray,very thin] (9.5,0) grid (16.5,3.5);
\draw[ultra thick,->] (9.5,0) -- (16.5,0);
\foreach \x in {10,11,12,13,14,15,16}
   \draw[thick] (\x cm,0.2cm) -- (\x cm,-0.2cm) node[anchor=north] {$\x$};

\draw[brown, very thick] (10,1) -- (12,1);
\fill[red] (10,1)  node[anchor=south ] {$\lambda_1$};
\fill[red] (12,1)  node[anchor=south west] {$\mu_1$};

\fill[brown] (10,1) circle (0.1cm);
\fill[brown] (12,1) circle (0.1cm);

\fill[red] (12,2)  node[anchor=south east] {$\lambda_3$};
\fill[red] (14,2)  node[anchor=south west] {$\mu_3$};

\fill[teal] (12,2) circle (0.1cm);
\fill[green] (13,2) circle (0.1cm);
\fill[brown] (14,2) circle (0.1cm);

\draw[purple,  very thick] (13,3) -- (15,3);
\fill[red] (13,3)  node[anchor=south east] {$\lambda_4$};
\fill[red] (16,3)  node[anchor=south west] {$\mu_5$};

\fill[purple] (13,3) circle (0.1cm);
\fill[purple] (15,3) circle (0.1cm);
\fill[green] (16,3) circle (0.1cm);

\end{tikzpicture}

\begin{tikzpicture}
\draw[step=1cm,gray,very thin] (9.5,0) grid (16.5,3.5);
\draw[ultra thick,->] (9.5,0) -- (16.5,0);
\foreach \x in {10,11,12,13,14,15,16}
   \draw[thick] (\x cm,0.2cm) -- (\x cm,-0.2cm) node[anchor=north] {$\x$};

\draw[green, very thick] (11,1) -- (12,1);
\fill[blue] (11,1)  node[anchor=south east] {$\lambda_2$};
\fill[blue] (13,1)  node[anchor=south west] {$\mu_2$};

\fill[brown] (13,1) circle (0.1cm);

\fill[green] (11,1) circle (0.1cm);
\fill[green] (12,1) circle (0.1cm);

\draw[green, very thick] (14,2) -- (15,2);
\fill[blue] (14,2)  node[anchor=south east] {$\lambda_5$};
\fill[blue] (15,2)  node[anchor=south west] {$\mu_4$};

\fill[green] (14,2) circle (0.1cm);
\fill[green] (15,2) circle (0.1cm);

\end{tikzpicture}

\end{minipage}

While $x_3 = (35124)$ is $321$-avoiding, one can be easily convinced that the pictorial ``tiling" exemplified above cannot be accomplished with that permutation. Thus, $\Pi(x_3)_\otimes$ does not appear in the Jacquet module of $\pi_1\times\pi_2$ and $x_3\not\in S(\pi_1,\pi_2)$.

The example discussed above belongs to a special family of cases, which we call regular. Namely, all $(\lambda_i)$ and $(\mu_i)$ satisfy strict inequalities. In the absence of such assumptions, Theorem \ref{thm-main} is still valid, yet it will deal with certain double-cosets inside $S_n$, rather than permutations.

For example, let us take $\gotM_1 = [10,10] + [11,12]$ and $\gotM_2= [11,11]$. Then, $(\lambda_1,\lambda_2,\lambda_3) = (10,11,11)$ and $(\mu_1,\mu_2,\mu_3) = (10,11,12)$.

\begin{center}
\begin{tikzpicture}
\draw[step=1cm,gray,very thin] (9.5,0) grid (12.5,3.5);
\draw[ultra thick,->] (9.5,0) -- (12.5,0);
\foreach \x in {10,11,12}
   \draw[thick] (\x cm,0.2cm) -- (\x cm,-0.2cm) node[anchor=north] {$\x$};

\fill[red] (10,1) circle (0.1cm) node[anchor=south ] {$\lambda_1=\mu_1$};

\fill[blue] (11,3) circle (0.1cm) node[anchor=south west] {$\lambda_3=\mu_2$};

\draw[red, very thick] (11,2) -- (12,2);
\fill[red] (11,2) circle (0.1cm) node[anchor=south east] {$\lambda_2$};
\fill[red] (12,2) circle (0.1cm) node[anchor=south west] {$\mu_3$};

\draw[red] (11.9,1.15) node {\Large $\gotM_1$};

\draw[blue] (10.3,3.3) node {\Large $\gotM_2$};

\end{tikzpicture}
\end{center}

Here, $x= (231)\in S_3$ will satisfy the Jacquet module condition. Yet, $\Pi(x) = \Pi(x')$, where $x'=(321)$. Since $x'$ in the longest representative of the double-coset involved and is not $321$-avoiding, $\Pi(x)$ will not appear as a subquotient of $\pi_1\times \pi_2$.

\begin{center}
\begin{tikzpicture}
\draw[step=1cm,gray,very thin] (9.5,0) grid (12.5,2.5);
\draw[ultra thick,->] (9.5,0) -- (12.5,0);
\foreach \x in {10,11,12}
   \draw[thick] (\x cm,0.2cm) -- (\x cm,-0.2cm) node[anchor=north] {$\x$};

\draw[brown, very thick] (10,1) -- (11,1);
\fill[brown] (10,1) circle (0.1cm) node[anchor=south east] {$\lambda_1$};
\fill[brown] (11,1) circle (0.1cm) node[anchor=south west] {$\mu_2$};

\draw[green, very thick] (11,2) -- (12,2);
\fill[green] (11,2) circle (0.1cm) node[anchor=south east] {$\lambda_2$};
\fill[green] (12,2) circle (0.1cm) node[anchor=south west] {$\mu_3$};

\draw (11,3) node {$\left([\lambda_3,\mu_2] = \emptyset\right)$};


\draw (11,-1) node[anchor=north] {\large $\gotM(x)$};


\end{tikzpicture}
\end{center}


\subsection{Quantum methods}
The key obstacle for proving our main theorem (Theorem \ref{thm-main}) is the determination of the content of Jacquet modules of irreducible representations. More precisely, we would like to know for which pairs $\sigma,\sigma'\in \irr$ the indicator representation $\sigma_\otimes$ can appear in the Jacquet module of $\sigma'$ (Theorem \ref{thm-main2}).

In order to obtain this sort of information we apply a quantization of the problem. The ring $\mathcal{R}$ (more precisely, a certain crucial subring of it, $\mathcal{R}^r$) can be viewed as a specialization at $q=1$ of the $\mathbb{Q}(q)$-algebra $U_q(\gotN)$, which is the positive part of the quantum group $U_q(\frak{sl}_{r+1})$ ($r$ here is a fixed large enough integer). Moreover, the basis of irreducible representations for $\mathcal{R}^r$ can be lifted to Lusztig's \textit{dual canonical basis} $\mathcal{B}$ (or, equivalently, Kashiwara's upper crystal basis) of $U_q(\gotN)$.

These natural identifications are a consequence of what was known as Zelevinski's \cite{zel-kl} $p$-adic Kazhdan-Lusztig conjecture, which was resolved through Ginzburg's geometric theory for affine Hecke algebras \cite{ginz-book}. The comparison between the geometries governing affine Hecke algebras and quantum groups, both of Lie type $A$, was explicated in the seminal paper of Ariki \cite{ariki-aff}. We will mostly rely on the survey of quantization results in \cite{LNT}, whose details are laid out in Sections \ref{sect:qu}, \ref{sect:qu2}.

We lift our problem to the quantum setting, that is, we look at products of two elements of $\mathcal{B}$, which correspond to ladder representations, inside a quantum group. If we write $b(\pi)\in\mathcal{B}$ for a lifting of a representation $\pi\in \irr$, then for ladder representations $\pi_1, \pi_2\in \irr$ we see (Section \ref{sec-multi}) an equation of the form
\[
b(\pi_1)b(\pi_2) = \sum_{x\in S(\pi_1,\pi_2)} q^{-d(\pi_2,\pi_1;\; \Pi(x))} b(\Pi(x))
\]
in $U_q(\gotN)$. By observing this, we are able to attach an integer invariant $d(\pi_1,\pi_2;\; \sigma)$ for each $\sigma\in \irr$ which appears as a subquotient in $\pi_1\times\pi_2$.

Furthermore, the algebra $U_q(\gotN)$ has a natural embedding into a \textit{quantum shuffle algebra}. The properties of that embedding and the behavior of $\mathcal{B}$ under it were thoroughly studied by Leclerc in \cite{leclerc-shuffle}. It was shown that this embedding can be seen as a quantization of the character morphism for affine Hecke algebras, which on the level of $p$-adic groups translates to the Jacquet functor.

We exploit the quantized character map to define yet another integer invariant $d_\otimes(\pi_1,\pi_2;\; \sigma)$, for every $\sigma\in \irr$ whose indicator $\sigma_\otimes$ appears as a subquotient in the Jacquet module of $\pi_1\times \pi_2$.

Assuming regularity conditions, we show that among $321$-avoiding permutations $x$ for which $d_\otimes(\pi_1,\pi_2;\; \Pi(x))$ is defined, there is a unique permutation $x_{\max}$ which attains the maximum value of $d_\otimes$. This uniqueness allows us to show, back in the classical setting, that for all $321$-avoiding permutations $x\neq x_{\max}$, $\Pi(x)_\otimes$ cannot appear in the Jacquet module of $\Pi(x_{\max})$ (Theorem \ref{thm-key}).

\subsection{Faster algorithm}\label{sect-intro-alg}

The algorithm supplied by Theorem \ref{thm-main} can quickly determine whether $\Pi(x)$ (keeping the notations from above) belongs to the product $\pi_1\times \pi_2$, for each $x\in S_n$. Yet, in order to produce a full decomposition one would still need to go over essentially all $321$-avoiding permutations in $S_n$.

We are able to bring a part of certain decomposition problems into an even lower complexity. Namely, when the multisegments that describe $\pi_1,\pi_2$ satisfy regularity conditions, i.e. when $\lambda_1 < \ldots <\lambda_n\leq \mu_1 < \ldots < \mu_n$ holds, Theorem \ref{thm-side} gives a fast algorithm that explicitly produces a subset of permutations $S^\circ(\pi_1,\pi_2)\subseteq S(\pi_1,\pi_2)$, without individually checking each potential subquotient.

Let us present briefly the algorithm by referring to a specific example. For that, we fix regular integers $\{\lambda_i,\mu_i\}_{i=1}^n$ and define ladder multisegments
\[
\gotM_1 =[\lambda_3 , \mu_2] + [\lambda_6 , \mu_7] + [\lambda_9 , \mu_9] + [\lambda_{10} , \mu_{10}] +[\lambda_{11} , \mu_{11}]\;,
\]
\[
\gotM_2 =   [\lambda_1 , \mu_1] + [\lambda_2 , \mu_3] + [\lambda_4 , \mu_4] + [\lambda_{5} , \mu_{5}] +[\lambda_{7} , \mu_{6}] +  [\lambda_{8} , \mu_{8}] \;,
\]

\begin{center}
\begin{tikzpicture}
\draw[step=0.5cm,gray,very thin] (0.75,0) grid (11.75,5.75);
\draw[ultra thick,->] (0.75,0) -- (11.75,0);
 \draw[thick] (1cm,0.2cm) -- (1cm,-0.2cm) node[anchor=north] {\footnotesize $\lambda_1$};
\draw[thick] (1.5cm,0.2cm) -- (1.5cm,-0.2cm) node[anchor=north] {\footnotesize $\lambda_2$};
\draw[thick] (2cm,0.2cm) -- (2cm,-0.2cm) node[anchor=north] {\footnotesize$\lambda_3$};
\draw[thick] (2.5cm,0.2cm) -- (2.5cm,-0.2cm) node[anchor=north] {\footnotesize $\lambda_4$};
\draw[thick] (3cm,0.2cm) -- (3cm,-0.2cm) node[anchor=north] {\footnotesize $\lambda_5$};
\draw[thick] (3.5cm,0.2cm) -- (3.5cm,-0.2cm) node[anchor=north] {\footnotesize $\lambda_6$};
\draw[thick] (4cm,0.2cm) -- (4cm,-0.2cm) node[anchor=north] {\footnotesize $\lambda_7$};
\draw[thick] (4.5cm,0.2cm) -- (4.5cm,-0.2cm) node[anchor=north] {\footnotesize $\lambda_8$};
\draw[thick] (5cm,0.2cm) -- (5cm,-0.2cm) node[anchor=north] {\footnotesize $\lambda_9$};
\draw[thick] (5.5cm,0.2cm) -- (5.5cm,-0.2cm) node[anchor=north] {\footnotesize $\lambda_{10}$};
\draw[thick] (6cm,0.2cm) -- (6cm,-0.2cm) node[anchor=north] {\footnotesize $\lambda_{11}$};

\draw[thick] (6.5cm,0.2cm) -- (6.5cm,-0.2cm) node[anchor=north] {\footnotesize $\mu_1$};
\draw[thick] (7cm,0.2cm) -- (7cm,-0.2cm) node[anchor=north] {\footnotesize $\mu_2$};
\draw[thick] (7.5cm,0.2cm) -- (7.5cm,-0.2cm) node[anchor=north] {\footnotesize $\mu_3$};
\draw[thick] (8cm,0.2cm) -- (8cm,-0.2cm) node[anchor=north] {\footnotesize $\mu_4$};
\draw[thick] (8.5cm,0.2cm) -- (8.5cm,-0.2cm) node[anchor=north] {\footnotesize $\mu_5$};
\draw[thick] (9cm,0.2cm) -- (9cm,-0.2cm) node[anchor=north] {\footnotesize $\mu_6$};
\draw[thick] (9.5cm,0.2cm) -- (9.5cm,-0.2cm) node[anchor=north] {\footnotesize $\mu_7$};
\draw[thick] (10cm,0.2cm) -- (10cm,-0.2cm) node[anchor=north] {\footnotesize $\mu_8$};
\draw[thick] (10.5cm,0.2cm) -- (10.5cm,-0.2cm) node[anchor=north] {\footnotesize $\mu_{9}$};
\draw[thick] (11cm,0.2cm) -- (11cm,-0.2cm) node[anchor=north] {\footnotesize $\mu_{10}$};
\draw[thick] (11.5cm,0.2cm) -- (11.5cm,-0.2cm) node[anchor=north] {\footnotesize $\mu_{11}$};

\draw[red, very thick] (1,0.5) -- (6.5,0.5);
\fill[red] (1,0.5) circle (0.1cm) ;
\fill[red] (6.5,0.5) circle (0.1cm) ;

\draw[red, very thick] (1.5,1) -- (7.5,1);
\fill[red] (1.5,1) circle (0.1cm) ;
\fill[red] (7.5,1) circle (0.1cm) ;

\draw[red, very thick] (2.5,2) -- (8,2);
\fill[red] (2.5,2) circle (0.1cm) ;
\fill[red] (8,2) circle (0.1cm) ;

\draw[red, very thick] (3,2.5) -- (8.5,2.5);
\fill[red] (3,2.5) circle (0.1cm) ;
\fill[red] (8.5,2.5) circle (0.1cm) ;

\draw[red, very thick] (4,3.5) -- (9,3.5);
\fill[red] (4,3.5) circle (0.1cm) ;
\fill[red] (9,3.5) circle (0.1cm) ;

\draw[red, very thick] (4.5,4) -- (10,4);
\fill[red] (4.5,4) circle (0.1cm) ;
\fill[red] (10,4) circle (0.1cm) ;

\draw[blue, very thick] (2,1.5) -- (7,1.5);
\fill[blue] (2,1.5) circle (0.1cm) ;
\fill[blue] (7,1.5) circle (0.1cm) ;

\draw[blue, very thick] (3.5,3) -- (9.5,3);
\fill[blue] (3.5,3) circle (0.1cm) ;
\fill[blue] (9.5,3) circle (0.1cm) ;

\draw[blue, very thick] (5,4.5) -- (10.5,4.5);
\fill[blue] (5,4.5) circle (0.1cm) ;
\fill[blue] (10.5,4.5) circle (0.1cm) ;

\draw[blue, very thick] (5.5,5) -- (11,5);
\fill[blue] (5.5,5) circle (0.1cm) ;
\fill[blue] (11,5) circle (0.1cm) ;

\draw[blue, very thick] (6,5.5) -- (11.5,5.5);
\fill[blue] (6,5.5) circle (0.1cm) ;
\fill[blue] (11.5,5.5) circle (0.1cm) ;


\draw[red] (9.45,1.75) node {\Large $\gotM_2$};

\draw[blue] (2.45,4.75) node {\Large $\gotM_1$};

\end{tikzpicture}
\end{center}

which will correspond to the ladder representations $\pi_1,\pi_2$. Let us demonstrate how to produce the subset $S^\circ(\pi_1,\pi_2)$ of subquotients, whose containment in $\pi_1\times \pi_2$ will not require running an algorithm over $C_{11}$ possible permutations.

We recall the permutation $\omega\in S_n$ ($n=11$ in the example) which is defined by $\gotM_1,\gotM_2$ as before. We also set $J_t,\, t=1,2$, to be the set of indices $i$, such that $[\lambda_i, \mu_{\omega(i)}]\in \gotM_t$.

In the example, $J_1 = \{3,6,9,10,11\}$, $J_2= \{1,2,4,5,7,8\}$.

We say that the triple $(\omega,J_1,J_2)$ is the \textit{combinatorial data} of the problem.

First, we put a linear order $\prec$ on $\{1,\ldots,n\}$ (see Section \ref{subs-comb}), which is increasing on both $J_1$ and $J_2$, and that every $i\in J_2$ is in the minimal position relative to $\prec$, for which $\lambda_i < \lambda_j$ and $\mu_{\omega(i)} < \mu_{\omega(j)}$ hold, for all $i\prec j\in J_1$.

In the example,
\[
{\color{red} 1}\;\prec\;{\color{blue}3}\;\prec\;{\color{red}2}\;\prec\;{\color{red}4}\;\prec\;{\color{red}5}\;\prec\;{\color{blue}6}\;\prec\;{\color{red}7}\;\prec\;{\color{red}8}\;
\prec\;{\color{blue}9}\;\prec\;{\color{blue}10}\;\prec\;{\color{blue}11}\quad.
\]

Now, let us construct inductively subsets $I_1\subset J_1$ and $I_2\subset J_2$ of equal size in the following manner. Pick $i\in J_2$, for which $j$, the next highest index after $i$ relative to $\prec$ is in $J_1$. In the example, we can choose $i = 8, j=9$. Add $i$ to $I_2$, $j$ to $I_1$ and proceed on the remaining indices.

One choice of construction of $I_1,I_2$ in our example can be through the following steps:
\[
{\color{red} 1}\;\prec\;{\color{blue}3}\;\prec\;{\color{red}2}\;\prec\;{\color{red}4}\;\prec\;{\color{red}5}\;\prec\;{\color{blue}6}\;\prec\;{\color{red}7}
\;\prec\;{\color{blue}10}\;\prec\;{\color{blue}11}\quad I_1 = \{9\},\, I_2 = \{8\}\;,
\]
\[
{\color{red} 1}\;\prec\;{\color{blue}3}\;\prec\;{\color{red}2}\;\prec\;{\color{red}4}\;\prec\;{\color{red}5}\;\prec\;{\color{blue}6}\;\prec\;{\color{blue}11}\quad I_1 = \{9,10\},\, I_2 = \{7,8\}\;,
\]
\[
{\color{red}2}\;\prec\;{\color{red}4}\;\prec\;{\color{red}5}\;\prec\;{\color{blue}6}\;\prec\;{\color{blue}11}\quad I_1 = \{3,9,10\},\, I_2 = \{1,7,8\}\;.
\]

Finally, we define $\omega_{I_1,I_2} = \omega\circ (3\; 1)(9\; 7)(10\; 8)$ by composing $\omega$ with disjoint transpositions taken in an ascending manner from the content of $I_1,I_2$ as in the example.

Theorem \ref{thm-side} says that $\omega_{I_1,I_2}\in S(\pi_1,\pi_2)$. The subset $S^\circ(\pi_1,\pi_2)$ is constructed by taking all the possible different choices of $I_1,I_2$.

In our example, one can count $14$ such choices, which give rise to $14$ irreducible subquotients of $\pi_1\times\pi_2$.

Moreover, under Conjecture \ref{conj-last} (see further discussion below), the set $S(\pi_1,\pi_2)^\circ$ will always contain the permutation $x_{\max}\in S_n$, for which $\Pi(x_{\max})$ is the unique irreducible sub-representation of $\pi_1\times\pi_2$. More precisely, $x_{\max}= \omega_{I^{\max}_1, I^{\max}_2}$, where $I^{\max}_1, I^{\max}_2$ is the choice of sets of maximal size.

In our example, $I^{\max}_1= \{3,6,9,10,11\},\; I^{\max}_2 = \{1,4,5,7,8\}$.






\subsection{Relation to KLR algebras}

The (positive part of the) quantum group $U_q(\gotN)$ was shown in \cite{KLdiag, rouq} to admit a monoidal categorification by graded finite-dimensional modules of what became known as Khovanov-Lauda-Rouquier algebras.

Specializing to our case of type $A$, this means the following. For each element $\alpha$ of the monoid $Q_+ = \oplus_{i=1}^r \mathbb{N}\epsilon_i$, there is a $\mathbb{Z}$-graded KLR-algebra $R_{\alpha}$\footnote{The definition of $R_\alpha$ includes an arbitrary base field $k$. In our discussion we will assume that $k$ is of characteristic $0$, in order for some stated results to hold.}.
Given two graded modules $M,N$ of $R_\alpha, R_\beta$, respectively, there is a convolution operation $M\circ N$ which gives a graded $R_{\alpha+\beta}$-module. Now, the $\mathbb{Q}(q)$-algebra $U_q(\gotN)$ is naturally $Q_+$-graded and has a $\mathbb{Z}[q,q^{-1}]$-form $U_{\mathcal{A}} = \oplus_{\alpha\in Q_+} U_\alpha$. Each $U_\alpha$ can be identified with the Grothendieck group of graded finite-dimensional modules of $R_\alpha$, in such way that that the convolution product is compatible with the quantum group product.

See \cite[Theorem 4.4]{kr2} for a convenient precise statement of the above.

Furthermore, it was shown \cite{vv} that elements of $\mathcal{B}$ are categorified precisely by the self-dual simple modules of the corresponding KLR algebras. The work in \cite{kr1} shows that basis elements which correspond to ladder representations, give in the KLR setting those simple modules that are homogeneous, i.e. concentrated at a single degree of the $\mathbb{Z}$-grading.

Thus, our results imply the following.

\begin{corollary}\label{cor-intro}
Suppose that $M,N$ are two simple self-dual homogenous modules of KLR algberas on the quiver $A_r$ which correspond (via the dual canonical basis) to $\pi_1,\pi_2\in \irr$, respectively. Then, the equality
\[
[N\circ M] = \sum_{x\in S(\pi_1,\pi_2)} q^{-d(\pi_2,\pi_1;\; \Pi(x))} [M(\Pi(x))]
\]
holds in the Grothendieck group of the corresponding KLR algebra.

Here $M(\Pi(x))$ denotes the self-dual simple module corresponding to $\Pi(x)\in \irr$ via $\mathcal{B}$. The action of $q$ on the isomorphism classes is by a shift of grading.
\end{corollary}

\subsection{Quotients of products}
Suppose now that $M,N$ are simple homogeneous self-dual KLR-algebra modules which correspond to a pair $\pi_1,\pi_2\in \irr$ which satisfies certain regularity conditions (a regular pair as defined in Section \ref{sec-ladd}).

We recall again that in this situation Theorem \ref{thm-key} gives an algorithm for a unique representation $\pi_{\max}= \Pi(x_{\max})\in \irr$ with $x_{\max}\in S(\pi_1,\pi_2)$, for which $d(\pi_1,\pi_2;\; \pi_{\max})$ is maximal.

Now, in light of Corollary \ref{cor-intro}, a result of McNamara \cite[Lemma 7.5]{McNm} (see also \cite[Section 4.2]{kkkoI}), which explicates a known property of geometric extension algebras (a class to which the KLR-algebras belong to by results of \cite{vv}), implies the following.

\begin{corollary}\label{cor-klr}
  The unique irreducible quotient of $N\circ M$ is given as $q^{-d(\pi_2,\pi_1;\; \pi_{\max})} M(\pi_{\max})$.
\end{corollary}


Back in the $p$-adic groups setting, it was shown in \cite{LM2} that a product of two ladder representations always admits a unique irreducible quotient. This is a property shared by a wider class of products involving what is known as \textit{square-irreducible} representations \cite{LM3}. The definition of that class was itself motivated by the analogous class of \textit{real} simple modules for KLR-algebras, which was extensively studied by Kang-Kashiwara-Kim-Oh and their school. In fact, the uniqueness statement in Corollary \ref{cor-klr} follows directly from \cite{kkko0}.

Moreover, Lapid-M\'inguez \cite[Corollary 6.16]{LM2} gave an explicit algorithm (involving the Zelevinski involution) for determining the irreducible quotient $\pi_{LM}$ of $\pi_2\times \pi_1$ (equivalently, the irreducible sub-representation of $\pi_1\times \pi_2$).

Although we limited the brief discussion above to the Grothendieck ring level, the wider picture includes exact tensor functors between the KLR setting and the affine Hecke algebras setting. The functors can be extracted either out of the KLR version of Schur-Weyl duality in \cite{kkkI}, or out of the type $A$ algebras isomorphisms in \cite{brun-kles, rouq}. While being well-known to experts, those issues will be further explicated in a forthcoming work of the author.

Once the existence of these functors is settled, we can deduce Proposition \ref{conj-last} from Corollary \ref{cor-klr}, which states that $\pi_{\max} \cong \pi_{LM}$.


The relation between the algorithms used to produce the representations $\pi_{\max}$ and $\pi_{LM}$ is intriguing and not fully understood. It received further attention in \cite{gur-lap}.

\subsection{Outline}
Section \ref{sect-nota} recalls the relevant basics of representation theory. In particular, it introduces the language of multisegments coming from the Langlands-Zelevinski classification, which is central to our discussion.

Section \ref{sec-back} portrays some known results and tools, such as the quantization of the rings in question. It also summarizes some results on indicator representations and ladder representations which were obtained in \cite{me-decomp}.

In Section \ref{sec-main} we state our main results.

Section \ref{sect-q} defines and exploits the quantum invariants for our decomposition problem. We apply some quantum shuffle algebra computations to obtain combinatorial formulae for these newly introduced invariants. Lemma \ref{lem-import} is a crucial technical step for obtaining the uniqueness of $x_{\max}$.

Section \ref{sec-case} gives the algorithm for producing the set of permutations $S^0(\pi_1,\pi_2)$ (and in particular, $x_{\max}$) out of the combinatorial data of a regular pair of ladder representations $\pi_1,\pi_2$.

Finally, Section \ref{sec-rs} deals with the remaining question of whether given a candidate $\pi\in\irr$, there are ladder representations $\pi_1,\pi_2\in \irr$ for which $\pi\cong \pi_{\max}$. Curiously, the answer lies in the Robinson-Schensted correspondence for permutations. 

\subsection{Acknowledgments}
Thanks are due foremost to Erez Lapid for being the driving force behind this project. I would also like to thank Bernard Leclerc for sharing some of the motivation in his papers, and Peter McNamara for some helpful discussions on KLR algebras.

The author is partially supported by the Israel Science Foundation (grant No. 737/20). This research was begun during the author's work in the Weizmann Institute of Science, Israel. That stay was supported by ISF grant No. 756/12, and ERC StG grant 637912.

\section{Preliminaries}\label{sect-nota}

\subsection{Generalities on representation theory}
For a $p$-adic group $G$, let $\mathfrak{R}(G)$ be the category
of smooth complex representations of $G$ of finite length. Denote by $\irr(G)$ the set of equivalence classes
of irreducible objects in $\mathfrak{R}(G)$. Denote by $\mathcal{C}(G)\subset \irr(G)$ the subset of irreducible supercuspidal representations. Let $\mathcal{R}(G)$ be the Grothendieck group of $\mathfrak{R}(G)$. We write $\pi\mapsto [\pi]$ for the canonical map $\mathfrak{R}(G) \to \mathcal{R}(G)$.

Given $\pi\in \mathfrak{R}(G)$, we have $[\pi] = \sum_{\sigma\in \irr(G)} c_\sigma\cdot [\sigma]$. For every $\sigma\in \irr(G)$, let us denote the multiplicity $m(\sigma, \pi):= c_{\sigma}\geq0$. For convenience we will sometimes write $m(\sigma, \pi)=0$ for representations $\pi,\sigma$ of two distinct groups.

Now, let $F$ be a fixed $p$-adic field. We write $G_n = GL_n(F)$, for all $n\geq1$, and $G_0$ for the trivial group.

For a given $n$, let $\alpha = (n_1, \ldots, n_r)$ be a composition of $n$. We denote by $M_\alpha$ the subgroup of $G_n$ isomorphic to $G_{n_1} \times \cdots \times G_{n_r}$ consisting of matrices which are diagonal by blocks of size $n_1, \ldots, n_r$ and by $P_\alpha$ the subgroup of $G_n$ generated by $M_\alpha$ and the upper
unitriangular matrices. A standard parabolic subgroup of $G_n$ is a subgroup of the form $P_\alpha$ and its standard Levi factor is $M_\alpha$. We write $\mathbf{r}_\alpha: \mathfrak{R}(G_n)\to \mathfrak{R}(M_\alpha)$ and $\mathbf{i}_\alpha: \mathfrak{R}(M_\alpha)\to \mathfrak{R}(G_n)$ for the \textit{normalized Jacquet functor} and the \textit{normalized parabolic induction} functor associated to $P_\alpha$.

Note that naturally $\mathcal{R}(M_\alpha)\cong \mathcal{R}(G_{n_1})\otimes \cdots\otimes \mathcal{R}(G_{n_r})$ and $\irr(M_\alpha)=\irr(G_{n_1})\times\cdots\times \irr(G_{n_r})$.

For $\pi_i\in \mathfrak{R}(G_{n_i})$, $i=1,\ldots,r$, we write
\[
\pi_1\times\cdots\times \pi_r := \mathbf{i}_{(n_1,\ldots,n_r)}(\pi_1\otimes\cdots\otimes \pi_r)\in \mathfrak{R}(G_{n_1+\ldots+n_r}).
\]
The image of a Jacquet functor applied on a representation will often be referred to as a Jacquet module of the representation.

Let us write $\mathcal{R} = \oplus_{m \geq 0} \mathcal{R}(G_m)$. This product operation defines a commutative ring structure on the group $\mathcal{R}$, where the trivial one-dimensional representation of $G_0$ is treated as an identity element.

We also write $\irr = \cup_{m\geq0} \irr(G_m)$ and $\mathcal{C} = \cup_{m\geq1} \mathcal{C}(G_m)$.

The following sections will sketch some necessary known facts about the ring $\mathcal{R}$ and the representations defining it. Most of the properties that will follow stem in the pivotal classification paper of Zelevinski \cite{Zel}. We recommend \cite[Section 3]{LM2} for a more recent overview.

\subsection{Supercuspidal lines}\label{sect:lines}

For every $\pi\in \irr$ there exist $\rho_1,\ldots,\rho_r \in\mathcal{C}$ for which $\pi$ is a sub-representation of $\rho_1\times\cdots\times \rho_r$. The notion of \textit{supercuspidal support} can then be defined as the set
\[
\supp(\pi)= \{\rho_1,\ldots,\rho_r\}\;.
\]
Note the difference in notation from \cite{me-decomp} in which the supercuspidal support was treated as a multi-set.

For any $n$, let $\nu^s= |\det|^s_F,\;s\in \mathbb{C}$ denote the family of one-dimensional representations of $G_n$, where $|\cdot|_F$ is the absolute value of $F$. For $\pi\in \mathfrak{R}(G_n)$, we write $\pi\nu^s := \pi\otimes \nu^s\in \mathfrak{R}(G_n)$.

Given $\rho\in \mathcal{C}$, we call
\[
\mathbb{Z}_{\langle \rho \rangle}:=\{\rho\nu^a\,:\;a\in \mathbb{Z}\}\subset \mathcal{C}
\]
the \textit{line} of $\rho$.

We write $\irr_{\langle \rho \rangle}\subset \irr$ for the collection of irreducible representations whose supercuspidal support is a subset of $\mathbb{Z}_{\langle \rho \rangle}$. We also write $\mathcal{R}_{\langle \rho \rangle}$ for the ring generated by $\irr_{\langle \rho \rangle}$ in $\mathcal{R}$.

Suppose that $\rho_1,\ldots , \rho_k\in\mathcal{C}$ are such that $\mathbb{Z}_{\langle \rho_1 \rangle}, \ldots, \mathbb{Z}_{\langle \rho_k \rangle}$ are pairwise disjoint sets. Then, for any collection $\{\pi_i \in \irr_{\langle \rho_i \rangle}\}_{i=1}^k$, $\pi_1 \times \cdots \times \pi_k$ is known to be irreducible.

Moreover, given $\pi \in \irr$, we can uniquely decompose it as $\pi= \pi_1 \times\cdots\times \pi_k$, with $\pi_i \in \irr_{\langle \rho_i \rangle}$, for such $\rho_1,\ldots,\rho_k\in \mathcal{C}$ whose lines are pairwise disjoint.

It follows that the study of the behavior of the set $\irr$ in the ring $\mathcal{R}$ can be reduced to the study of $\irr_{\langle \rho \rangle}$ in the ring $\mathcal{R}_{\langle \rho \rangle}$, for a single $\rho\in \mathcal{C}$. More precisely, from the above discussion we clearly have
\[
m(\sigma_1\times \cdots\times \sigma_k, (\pi^1_1\times \cdots\times \pi^1_k) \times (\pi^2_1\times \cdots\times \pi^2_k))
= \prod_{i=1}^k m(\sigma_i, \pi^1_i\times\pi^2_i)\;,
\]
where $\pi^1_i,\pi^2_i,\sigma_i\in \irr_{\langle \rho_i \rangle}$ are such that the lines of  $\rho_1,\ldots,\rho_k\in \mathcal{C}$ are pairwide disjoint.

For most of this work it will suffice to take $\rho_0\in \mathcal{C}$ to be the trivial representation of $G_1$ and study $\irr_0: = \irr_{\langle \rho_0 \rangle}$ in $\mathcal{R}_0:= \mathcal{R}_{\langle \rho_0 \rangle}$.

For every $r\in\mathbb{N}$, we write $\mathcal{R}^r$ for the subring of $\mathcal{R}_0$ generated by
\[
\irr^r:=\left\{ [\pi]\;:\; \pi\in \irr,\, \supp(\pi)\subset\{\nu^1,\ldots,\nu^r\} \right\}\;.
\]
Hence, $\mathcal{R}^1 \subset \mathcal{R}^2\subset \mathcal{R}^3\subset \ldots$ and  $[\nu^{a(\pi)}\pi]\in \cup_{r\geq 1} \mathcal{R}^r$ for every $[\pi]\in \mathcal{R}_0$ and large enough integer $a(\pi)$. We also write $\mathcal{R}^r_{\mathbb{Q}} := \mathbb{Q}\otimes_\mathbb{Z} \mathcal{R}^r$.

\subsection{Multisegments}
We will first describe a version of the Langlands classification of $\irr$, in the terms of \cite{Zel} with adaptations for our needs.

Given $\rho\in \mathcal{C}(G_n)$ and two integers $a\leq b$, we write $L([a,b]_\rho)\in \irr(G_{n(b-a+1)})$ for the unique irreducible quotient of $\rho\nu^{a}\times \rho\nu^{a+1}\times\cdots\times \rho\nu^b$. It will also be helpful to set formally $L([a,a-1]_\rho)$ as the trivial representation of $G_0$.

When $\rho=\rho_0$ (the trivial representation of $G_1$), we will simply write $L([a,b]) = L([a,b]_{\rho_0})\in \irr_0$.

We also treat the \textit{segment} $\Delta= [a,b]_\rho$ as a formal object defined by a triple $([\rho],a,b)$. We denote by $\seg_0$ the collection of all segments $[a,b]:= [a,b]_{\rho_0}$ that are defined by $\rho_0$ and integers $a-1\leq b$.

For a segment $\Delta = [a,b]\in \seg_0$, we will write $a = b(\Delta)$ and $b= e(\Delta)$.

More generally, we denote the larger collection $\seg_0\subset \seg$ of all segments that are defined by any $\rho\in \mathcal{C}$ and integers $a-1\leq b$, up to the equivalence $[a,b]_\rho=[a',b']_{\rho'}$, when $\rho\nu^a\cong \rho'\nu^{a'}$ and $\rho\nu^b \cong \rho'\nu^{b'}$.

A segment $\Delta_1$ is said to precede a segment $\Delta_2$, if $\Delta_1 = [a_1,b_1]_{\rho} ,\;\Delta_2= [a_2,b_2]_{\rho}$ and $a_1\leq a_2-1\leq b_1<b_2$. We will write $\Delta_1 \prec \Delta_2$ in this case and say that the pair $\{\Delta_1,\Delta_2\}$ is linked.

We will write $[a_1,b_1]_{\rho}\subseteq[a_2,b_2]_{\rho}$ when $a_2\leq a_1$ and $b_1\leq b_2$.

Given a set $X$, we write $\mathbb{N}(X)$ for the commutative semigroup of maps from $X$ to $\mathbb{N}= \mathbb{Z}_{\geq0}$ with finite support. The elements of $\mathbb{N}(\seg)$ are called \textit{multisegments}.

Naturally, we will sometimes write a single segment $\Delta\in \seg$ as an element of $\mathbb{N}(\seg)$. By doing so we technically refer to the indicator function of the segment.

The statement of the \textit{Langlands classification} for the groups $\{G_n\}_{n=1}^\infty$ can be stated as a bijection (see, for example, \cite[Theorem 3.6]{LM2}
\[
L: \mathbb{N}(\seg) \to \irr\;,
\]
that extends the definition of $L$ for a single segment described above.

The map $L$ can also be restricted to a bijection
\[
L: \mathbb{N}(\seg_0) \to \irr_0\;.
\]

As discussed in \cite{me-decomp}, it is often useful to parameterize multisegments by means of permutations, as follows.

We write $S_n$ for the group of permutations on $\{1,\ldots,n\}$.

Let $\mathcal{P}_n$ denote the collection of tuples $(\lambda_1,\ldots,\lambda_n)\in\mathbb{Z}^n$, with $\lambda_1\leq \lambda_2 \leq \ldots \leq \lambda_n$.

Let $\lambda = (\lambda_1,\ldots,\lambda_n)$, $\mu = (\mu_1,\ldots,\mu_n)\in \mathcal{P}_n$ be given. Let $w\in S_n$ be a permutation for which $\lambda_i \leq \mu_{w(i)}+1$ holds for all $1\leq i\leq n$. For such parameters we define multisegments
\[
\gotM^w_{\lambda,\mu} = \sum_{i=1}^n [\lambda_i, \mu_{w(i)}]\in \mathbb{N}(\seg_0)\;,\quad\mbox{and }  \gotM^{w,\rho}_{\lambda,\mu} = \sum_{i=1}^n [\lambda_i, \mu_{w(i)}]_\rho\in \mathbb{N}(\seg)\;,
\]
for every $\rho\in \mathcal{C}$.

We write $Q(\lambda, \mu) \subset S_n$ for the set of permutations $w$, for which $\gotM^w_{\lambda,\mu}$ is defined (i.e. $\lambda_i \leq \mu_{w(i)}+1\;\forall i$).

\section{Background}\label{sec-back}

\subsection{The character morphism}
Let us fix an integer $r\geq 1$. We write $\mathcal{I}= \{\epsilon_1,\ldots,\epsilon_r\}$ as a set of formal letters, and let $\mathcal{M}$ be the free monoid of words in the alphabet $\mathcal{I}$.

For all $1\leq s\leq t\leq r$, we denote the word $\epsilon(s,t):= \epsilon_t\epsilon_{t-1}\cdots \epsilon_s \in \mathcal{M}$. An expression of the form $\epsilon(s,s-1)$ will denote the empty word.

For a ring $R$, we let $\mathcal{F}_R$ to be the free $R$-module formally spanned by the basis $\mathcal{M}$. In fact, $\mathcal{F}_R$ has a natural algebra structure coming from the  product operation (concatenation of words) in $\mathcal{M}$, which we will denote simply as a product.

The resulting algebra is naturally graded by the commutative monoid $Q_+ = \oplus_{i=1}^r \mathbb{N} \epsilon_i$. For a word $w\in \mathcal{M}$, we write $|w|\in Q_+$ to be its degree as an element of $\mathcal{F}_{\mathbb{Z}}$.

We also equip $\mathcal{F}_R$ with another algebra structure coming from the \textit{shuffle product} $\shuffle$. This product is defined inductively on basis elements (induction on the length parameter of a word) as
\[
w\epsilon \shuffle x\delta = (w\epsilon \shuffle x)\delta + (w\shuffle x\delta)\epsilon\,
\]
for all $w,x\in \mathcal{M}$ and $\epsilon,\delta\in \mathcal{I}$.

Note that the resulting shuffle product is commutative.

For $x\in \mathcal{F}_R$ and $w\in \mathcal{M}$, we denote by $D_w^R(x)\in R$ the coefficient of $w$ in the $\mathcal{M}$-expansion of $x$, i.e. $x = \sum_{w\in \mathcal{M}} D_w^R(x) w$.

Let $\alpha_0 = (1,1,\ldots,1)$ be a composition of an integer $n\geq1$. Then, for every representation $\pi\in \mathfrak{R}(G_n)$ such that $[\pi ] \in \mathcal{R}^r$, we have an expansion
\[
[\mathbf{r}_{\alpha_0}(\pi)] = \sum_{i=1}^k [\nu^{a^1_i}\otimes\cdots \otimes \nu^{a^n_i}]\,
\]
in $\mathcal{R}(M_{\alpha_0})$, for some integers $1\leq a^j_i\leq r$.

We write
\[
\ch(\pi) = \sum_{i=1}^k \epsilon_{a^1_i} \cdots \epsilon_{a^n_i}\in \mathcal{F}_{\mathbb{Z}}.
\]
As a consequence of the Geometric Lemma of Bernstein-Zelevinski (see, for example \cite{LM2}, or \cite[Lemma 2.7]{groj-vaz}), the resulting map $ch:\mathcal{R}^r \to \mathcal{F}_{\mathbb{Z}}$ becomes a (commutative) ring homomorphism with respect to the shuffle product on $\mathcal{F}_{\mathbb{Z}}$.

Sometimes we will also write $\ch$ for the extended map $id\otimes \ch:  \mathcal{R}^r_{\mathbb{Q}} \to \mathcal{F}_{\mathbb{Q}}$.

Let $M_\alpha< G$ be a Levi subgroup, and let $\sigma = \sigma^1\otimes \cdots \otimes \sigma^k \in \irr(M_\alpha)$ be a representation, such that $[\sigma^i]\in \mathcal{R}^r$, for all $1\leq i\leq k$. We can then naturally extend the definition of the character morphism by setting
\[
\ch_\alpha(\sigma):=\ch(\sigma^1) \ch(\sigma^2)\cdots \ch(\sigma^k)\;.
\]
As a consequence of the transitivity property of Jacquet functors, the equality $\ch = \ch_\alpha\circ \mathbf{r}_\alpha$ of maps from $\mathcal{R}^r$ will hold.

\subsection{Quantization of the character morphism}\label{sect:qu}
Let us expand our discussion to a quantized version of the ring $\mathcal{R}^r$ and a $q$-analog of the character map. We will follow the description portrayed in \cite[6.8]{leclerc-shuffle}.

We first put a symmetric bilinear form $(\,,\,)$ on $Q_+$ by setting $(\epsilon_i,\epsilon_j) = 2\delta_{i,j} - \delta_{i,j+1} - \delta_{i,j-1}$ and extending linearly. This form is reminiscent of the Cartan bilinear form of the root data of the simple Lie algebra $\mathfrak{sl}_{r+1}$.

For the field of rational functions $\mathbb{Q}(q)$, we equip $\mathcal{F}_{\mathbb{Q}(q)}$ with yet another algebra structure $\ast$, to make it into a (non-commutative) \textit{quantum shuffle algebra}. This product is defined inductively on basis elements as
\[
w\epsilon \ast x\delta = (w\epsilon \ast x)\delta + q^{-(|x\delta|,\epsilon)} (w\ast x\delta)\epsilon\,
\]
for all $w,x\in \mathcal{M}$ and $\epsilon,\delta\in \mathcal{I}$.

Note, that this is the opposite product than the one used in \cite{leclerc-shuffle}.

Recall the quantum group $U_q(\mathfrak{n})$, which is a $\mathbb{Q}(q)$-algebra attached to the nilpotent part $\mathfrak{n}$ of $\mathfrak{sl}_{r+1}$ (see \cite{leclerc-shuffle}), for example, for the definition).

It is known (\cite[Theorem 4]{leclerc-shuffle}) that there is an embedding $\Phi:U_q(\mathfrak{n}) \to \mathcal{F}_{\mathbb{Q}(q)}$, which becomes an anti-homomorphism of algebras with respect to our conventions for the quantum shuffle product.

This quantum group together with $\Phi$ serves as a quantization of $\mathcal{R}^r$ and its character map in the following sense.

Let $\mathcal{A} = \mathbb{Z}[q,q^{-1}]\subset \mathbb{Q}(q)$ be a subring. There is an $\mathcal{A}$-subalgebra $U_{\mathcal{A}}\subset U_q(\mathfrak{n})$ and an $\mathcal{A}$-subalgebra $\mathcal{F}^{\mathcal{A}}\subset \mathcal{F}_{\mathbb{Q}(q)}$, so that
\begin{equation}\label{eq1}
\Phi(U_{\mathcal{A}}) = \Phi(U_q(\mathfrak{n})) \cap \mathcal{F}^{\mathcal{A}}\;,
\end{equation}
\begin{equation}\label{eq2}
\mathcal{R}^r_{\mathbb{Q}}\cong \mathbb{Q}\otimes_{\mathcal{A}}U_{\mathcal{A}}\;,
\end{equation}
\begin{equation}\label{eq3}
 \mathcal{F}_{\mathbb{Q}}\cong \mathbb{Q}\otimes_{\mathcal{A}} \mathcal{F}^{\mathcal{A}}
\end{equation}
hold, where $\mathbb{Q}$ is considered as an $\mathcal{A}$-module by evaluation in $q=1$. Equation \eqref{eq1} is \cite[Lemma 8]{leclerc-shuffle}. The natural algebra isomorphism in \eqref{eq2} is a consequence of \cite[Proposition 7]{LNT} together with the discussion in \cite[3.5]{LNT} (see also \cite{bz-strings}), while \eqref{eq3} is discussed in \cite[2.8.1]{leclerc-shuffle} as well.

We will write $S: U_{\mathcal{A}} \to \mathcal{R}^r_\mathbb{Q}$ or $S: \mathcal{F}^{\mathcal{A}} \to \mathcal{F}_{\mathbb{Q}}$ for the resulting specialization maps. The nature of the isomorphism in \eqref{eq3} is such that for any $x\in \mathcal{F}^{\mathcal{A}}$ and any word $w\in\mathcal{M}$, we have $\lim_{q\to 1} D^{\mathbb{Q}(q)}_w(x) = D^{\mathbb{Q}}_w(S(x))$.
The resulting diagram is commutative in the sense that
\[
S\circ \Phi|_{U_{\mathcal{A}}}= ch\circ S
\]
on $U_\mathcal{A}$, as shown in \cite[6.8]{leclerc-shuffle}.

\subsection{Dual canonical basis}\label{sect:qu2}
The $\mathcal{A}$-ring $U_{\mathcal{A}}$ has a distinguished basis $\mathcal{B}$ called the \textit{dual canonical basis}. Together with its dual version -- the canonical basis, it is known to play a meaningful role in the theory of quantum groups. Among its special properties is the fact that $\mathcal{B}$ descends to the basis of irreducible representations when specialized at $q=1$. In other words, we have
\[
S(\mathcal{B})= \irr^r\;.
\]
Thus, $S$ gives a bijection between $\mathcal{B}$ and the irreducible representations in $\mathcal{R}^r$. See \cite[Theorem 12]{LNT} for one possible proof and discussion of that property.

Since elements of $\irr$ can be referred to by multisegments through the Langlands classification, it will be convenient to extend this notation to the elements of the dual canonical basis. For every $\gotM\in \seg_0$ such that $L(\gotM)\in \irr^r$, we will write $b(\gotM)\in\mathcal{B}$ for the element that satisfies $\{b(\gotM)\} = \mathcal{B} \cap S^{-1}(L(\gotM))$. For a representation $\pi \in \irr^r$, we will write $b(\pi):= b(L^{-1}(\pi))$.

We recall some results of Leclerc (\cite[Proposition 39, Theorem 42]{leclerc-shuffle} ) which will be of importance to this work.
\begin{proposition}\label{leclerc-thm}
For every $b\in \mathcal{B}$ and $w\in \mathcal{M}$, we have
\[
D^{\mathbb{Q}(q)}_w(\Phi(b))\in \mathbb{N}[q,q^{-1}]\;.
\]
Moreover, the above polynomials are all symmetric in $q$ and $q^{-1}$.
\end{proposition}

\subsection{Indicator representations}\label{sect-ind}
Let us recall the concept of an indicator representation which was introduced in \cite{me-decomp}.

Let $\pi\in \irr(G_n)\cap \irr_{\langle \rho \rangle}$ be a given representation, and $\gotM\in \mathbb{N}(\seg)$ be the multisegment for which $\pi= L(\gotM)$.

We can write $\gotM = \sum_{i=1}^k [a_i,b_i]_\rho$, so that $a_i\leq b_i$. We can further suppose that our indexing satisfies $a_1\leq \ldots \leq a_k$, and when $a_i = a_{i+1}$ we have $b_i\leq b_{i+1}$.
Moreover, let us take note of the indices $1\leq i_1<\ldots <i_t=k$, such that
\[
a_1= \ldots= a_{i_1} < a_{i_1+1} = \ldots = a_{i_2}< \ldots \leq a_{i_t}\;.
\]

We then can write $\gotM = \gotM_1+\cdots+ \gotM_t$, where
\[
\gotM_j= \sum_{i= i_{j-1}+1}^{i_j} [a_i,b_i]\;,
\]
with $i_0=0$. In other words, we separate the multisegment into pieces clustered according to the begin points of the segments.

With these notations we define the following.

\begin{definition}\label{defi:indi}
For $\pi\in \irr(G_n)\cap \irr_{\langle \rho \rangle}$, \textit{the indicator representation} of $\pi$ is
\[
\pi_\otimes = L\left(\gotM_1\right)\otimes \cdots\otimes L\left(\gotM_t\right)\in \irr(M_{\alpha_\pi})\;,
\]
where $M_{\alpha_\pi}<G_n$ is the corresponding Levi subgroup.
\end{definition}

It is easily seen that $m(\pi_\otimes, \mathbf{r}_{\alpha_\pi}(\pi))>0$, for $\pi\in \irr_{\langle \rho \rangle}$ (see \cite[Section 3.1]{me-decomp} for details). Thus, by exactness of the Jacquet functor $\mathbf{r}_{\alpha_\pi}$, we see a useful property, that for all $\sigma\in \mathfrak{R}(G_n)$, the inequality
\begin{equation}\label{maj}
m(\pi,\sigma) \leq m(\pi_\otimes, \mathbf{r}_{\alpha_\pi}(\sigma))
\end{equation}
holds.

Let us now additionally assume that $[\pi]\in \mathcal{R}^r$. We would like to exhibit a simple relation of $\pi_\otimes$ with word expansions of characters of representations in $\mathcal{R}^r$.

We can define
\[
\epsilon(\pi) = \epsilon(\gotM) = \epsilon(a_1,b_1)\epsilon(a_2,b_2)\cdots \epsilon(a_k,b_k)\;.
\]

\begin{lemma}\label{lemma-first}
Suppose that $a_1<\ldots<a_k$. For any $\sigma\in \irr^r$, we have

\[
m(\pi_\otimes, \mathbf{r}_{\alpha_\pi}(\sigma)) = D^{\mathbb{Q}}_{\epsilon(\pi)}(\ch(\sigma))\;.
\]

\end{lemma}

\begin{proof}
Let $\nu(\pi)$ be the representation of the minimal Levi of $G_n$ for which $\ch_{\alpha_0}(\nu(\pi)) = \epsilon(\pi)$. The right hand side counts the multiplicity of $\nu(\pi)$ in $\mathbf{r}_{\alpha_0}(\sigma)$. From transitivity of Jacquet modules, this equivalent to counting the multiplicity of $\nu(\pi)$ in the minimal Jacquet module of $\mathbf{r}_{\alpha_\pi}(\sigma)$.

Let $\tau$ be an irreducible subquotient of $\mathbf{r}_{\alpha_\pi}(\sigma)$. If $\tau\cong \pi_\otimes$, then clearly $\nu(\pi)$ appears with multiplicity one in $\mathbf{r}_{\alpha_0}(\tau)$. It remains to show that when $\tau\not\cong \pi_\otimes$, the representation $\nu(\pi)$ does not appear in $\mathbf{r}_{\alpha_0}(\tau)$.

We assume the contrary. By assumption $\tau = \tau_1\otimes \cdots \otimes \tau_k$, with some $\tau_i\not\cong L([a_i,b_i])$. Also, by assumption $\mathbf{r}_{\alpha_0}(\tau_i)$ contains $\nu^{b_i}\otimes\nu^{b_i-1}\otimes \cdots \otimes \nu^{a_i}$. Since the latter is the unique irreducible representation having its supercuspidal support, we can also assume it appears as a quotient. By reciprocity $\tau_i$ is contained in $\nu^{b_i}\times\nu^{b_i-1}\times \cdots \times \nu^{a_i}$. Yet, as a simple consequence of the Gelfand-Kazhdan duality (see \cite[Proposition 2.2]{me-restriction}) and the definition of segment representations, the socle of this product is $L([a_i,b_i])$, which gives a contradiction.

\end{proof}

\subsection{Ladders}\label{sec-ladd}
We will call a tuple $\lambda= (\lambda_1, \ldots,\lambda_n)\in \mathcal{P}_n$ \textit{regular} if it admits the strict inequalities $\lambda_1< \ldots <\lambda_n$.

A \textit{ladder representation} $\pi\in \irr$ is an irreducible representation which can be written as $\pi=L(\gotM^{e,\rho}_{\lambda, \mu})$ for some regular tuples $\lambda,\mu\in \mathcal{P}_n$, $e\in S_n$ the identity permutation and a supercuspidal $\rho\in \mathcal{C}$.

In other words, $\pi = L\left(\sum_{i=1}^n [a_i,b_i]_\rho\right)$ for some $a_1<\ldots <a_n$ and $b_1<\ldots <b_n$.

We will call a pair of ladder representations $\pi_1,\pi_2\in \irr_0$ \textit{regular}, if we can write $\pi_1= L(\gotM_1)$, $\pi_2 = L(\gotM_2)$, $\gotM_1+ \gotM_2 = \gotM_{\lambda,\mu}^w$, for regular $\lambda= (\lambda_i),\,\mu=(\mu_i)\in \mathcal{P}_n$ which satisfy $\lambda_n\leq \mu_1$, and $w\in S_n$.

In particular, $Q(\lambda,\mu)=S_n$ for a regular pair of ladders as above. It is also clear that $w$ is determined uniquely by the regular pair.

We take note of the partition $J_1\cup J_2 =  \{1,\ldots, n\}$, for which $\gotM_1 = \sum_{i\in J_1} [\lambda_i, \mu_{w(i)}]$ and $\gotM_2 = \sum_{i\in J_2} [\lambda_i, \mu_{w(i)}]$. From the definition of ladders, the restrictions $w|_{J_1},\, w|_{J_2}$ are increasing functions.

We will refer to $(w,J_1, J_2)$ as the \textit{combinatorial data} of the regular pair of ladder representations $\pi_1,\pi_2$.

\subsection{Indicator representations and ladders}
Let us survey the consequences of \cite[Section 5]{me-decomp} on the appearance of indicator representations in Jacquet modules of products of two ladder representations.

Suppose that $\pi_1,\pi_2\in \irr_0$ are two ladder representations. Suppose that $\sigma\in \irr$ is such that $m(\sigma_\otimes, \mathbf{r}_{\alpha_\sigma}(\pi_1\times \pi_2))>0$. Then, $m(\sigma_\otimes, \mathbf{r}_{\alpha_\sigma}(\pi_1\times \pi_2))=1$.

Moreover, there are unique Jacquet modules $\mathbf{r}_{\beta_i}$ and unique irreducible representations $[\sigma_i]\leq [\mathbf{r}_{\beta_i}(\pi_i)],\;i=1,2$, for which $\sigma_\otimes$ appears as a subquotient of $\sigma_1\times\sigma_2$, for
\[
\sigma_1\times \sigma_2:= (\sigma_1^k\times \sigma_2^k) \otimes\cdots\otimes (\sigma_1^1\times \sigma_2^1) \in \mathfrak{R}(M_{\alpha_\sigma})\;,
\]
where
\[
\sigma_i = \sigma_i^k\otimes\cdots\otimes \sigma_i^1 \in \mathfrak{R}(M_{\beta_i})\;,\quad i=1,2\;.
\]

In fact, $m(\sigma_\otimes, \sigma_1\times\sigma_2)=1$.

The following proposition will put some of the observations of \cite[Section 5]{me-decomp} into a more explicit form.

\begin{proposition}\label{prop-matrix}
Suppose that $\pi_1 = L(\gotM_1),\,\pi_2 = L(\gotM_2)\in \irr_0$ is a regular pair of ladder representations, with combinatorial data $(w,J_1,J_2)$ and $\gotM_1+ \gotM_2 = \gotM_{\lambda,\mu}^w$.

Let $\sigma = L(\gotN)\in \irr_0$ be given.

Then, $m(\sigma_\otimes, \mathbf{r}_{\alpha_\sigma}(\pi_1\times \pi_2))>0$ holds, if and only if, it is possible to write $\gotN= \gotM_{\lambda,\mu}^x$ for a permutation $x\in S_n$,
and a matrix of integers
\[
C = C(\pi_1,\pi_2;\;\sigma)= C(\gotM_1,\gotM_2;\;\gotN)=(c_i^j)_{i=1,\ldots, n}^{j=1,\ldots, n+1}
\]
exists, which satisfies the following list of properties:
\begin{enumerate}
  \item\label{cond-one}  For every $1\leq i\leq n$, $\lambda_i = c_i^1\leq c_i^2\leq \ldots \leq c_i^{n+1} = \mu_{w(i)}+1$.
  \item\label{cond-two}  For every $1\leq j\leq n$ and all $i_1<i_2$ such that $i_1,i_2\in J_1$ or $i_1,i_2\in J_2$, we have $c^{j+1}_{i_1} < c^j_{i_2}$.
  \item\label{cond-three}  For every $1\leq j\leq n$, the supercuspidal supports satisfy the equality
  \[
  \bigcup_{i=1}^n \supp([c^j_i, c^{j+1}_i-1]) = \supp([\lambda_{n+1-j}, \mu_{x(n+1-j)}])\;,
  \]
   with the union being disjoint.

   In other words, for every $1\leq j \leq n$, there is a sequence of indices $i_1,\ldots, i_t$, for which $c^{j+1}_{i_1} =\mu_{x(n+1-j)}+1,\, i_t = n+1-j$ and $c^j_{i_s} = c^{j+1}_{i_{s+1}}> c^j_{i_{s+1}}$ hold, for all $1\leq s<  t$, and $c^{j+1}_i = c^{j}_i$, for all $i\not\in\{i_1,\ldots, i_t\}$.
\end{enumerate}

When these conditions hold, the irreducible representations $\sigma_1= \otimes_j\sigma_1^j,\,\sigma_2= \otimes_j\sigma_2^j$ described above are given as
 \[
\sigma_1^j = \times_{i\in J_1} L([c^j_i, c^{j+1}_i-1]), \quad \sigma_2^j = \times_{i\in J_2} L([c^j_i, c^{j+1}_i-1])\;,
\]
for every $1\leq j\leq n$.

 \end{proposition}

\begin{proof}
Suppose that $m(\sigma_\otimes, \mathbf{r}_{\alpha_\sigma}(\pi_1\times \pi_2))>0$. Let us write $\sigma_\otimes = \tau_k\otimes \cdots\otimes \tau_1$.

It follows directly from the Jacquet module description of ladder representations in \cite{LapidKret}, that a matrix $C= (c_i^j)_{i=1,\ldots, n}^{j=1,\ldots, k+1}$ of integers exists, so that
 \[
\sigma_1^j = L\left(\sum_{i\in J_1} [c^j_i, c^{j+1}_i-1]\right), \quad \sigma_2^j = L\left(\sum_{i\in J_2} [c^j_i, c^{j+1}_i-1]\right)\;,
\]
for every $1\leq j\leq k$. It also follows from same description that condition (\ref{cond-one}) holds (with $n$ replaced by $k$ at the moment), and that for every $1\leq j\leq k$ and all $i_1<i_2$ such that $i_1,i_2\in J_1$ or $i_1,i_2\in J_2$, we have $c^{j}_{i_1} < c^j_{i_2}$.

By \cite[Proposition 5.1]{me-decomp}, $\sigma_1^j,\sigma_2^j$ are all generic representations, meaning that the multisegments defining them consist of pairwise unlinked segments. The disjointness property in condition (\ref{cond-three}) and the rest of condition (\ref{cond-two}) now follow.

Note, that the lowermost point in the supercuspidal support of $\pi_1\times\pi_2$ is $\nu^{\lambda_1}$. Moreover, because of the regularity condition it appears in the supercuspidal support with multiplicity one. It follows from definition of indicator representations that $\tau_k \cong L([\lambda_1,\zeta_1])$. More specifically, $\zeta_1 = c^{n+1}_s -1 = \mu_{w(s)}$ for a certain index $1\leq s\leq n$. We set $x(1):= w(s)$.

Arguing inductively on each $1\leq j\leq k$, we see that $k=n$ and that $\tau_j \cong L([\lambda_{n+1-j}, \mu_{x(n+1-j)}])$, for a permutation $x\in S_n$.

The inverse implication follows easily from similar arguments, once one makes the observation that $L([\lambda_{n+1-j}, \mu_{x(n+1-j)}])$ appears as a subquotient in the product
\[
\times_{i=1}^n L([c^j_i, c^{j+1}_i-1])
\]
defined by segments satisfying condition (\ref{cond-three}).

\end{proof}

\begin{remark}\label{rmrk-matrix}
For future reference let us make the observation that for every entry $(i,j)$ in a matrix $C$ as in the proposition above, for which $c_i^1<c_i^j$ holds, we must have $c^j_i= \mu_t+1$, for some $t$.
\end{remark}

%

\subsubsection{Example}\label{exampl}
Let us recall the example considered in the introduction section, in which the multisegments were taken as
\[
\gotM_1 = [10,12] + [12,14] + [13,16],\quad\gotM_2 = [11,13]+[14,15]\;,
\]
so that $\gotM_1 + \gotM_2= \gotM_{\lambda,\mu}^\omega$, where
\[
(\lambda_1,\ldots, \lambda_5) = (10,11,12,13,14),\; (\mu_1,\ldots, \mu_5) = (12,13,14,15,16),
\]
\[
J_1 = \{1,3,4\},\, J_2= \{2,5\},\,\omega = (12354)\in S_5\;.
\]
It was shown pictorially that $\sigma = L(\gotM_{\lambda,\mu}^{x_1})$, where $x_1 = (34152)$, satisfies $m(\sigma_\otimes, \mathbf{r}_{\alpha_\sigma}(L(\gotM_1)\times L(\gotM_2)))>0$. In terms of the matrices described in Proposition \ref{prop-matrix}, the pictures correspond to the matrix
\[
C(\gotM_1,\gotM_2;\;\gotM_{\lambda,\mu}^{x_1}) = \left( \begin{array}{ccc}
                                                          c_1^1 & \ldots & c_1^6 \\
                                                           & \ldots & \\
                                                          c_5^1 & \ldots & c_5^6
                                                        \end{array} \right) =
\left( \begin{array}{cccccc}
                                                          10 & 10 & 10 & 10 & 10 & 13 \\
                                                          11 & 11 & 11 & 11 & 13 & 14 \\
                                                          12 & 12 & 12 & 13 & 14 & 15 \\
                                                          13 & 13 & 17 & 17 & 17 & 17 \\
                                                          14 & 14 & 14 & 14 & 16 & 16
                                                        \end{array} \right)\;.
\]

\section{Main Problem}\label{sec-main}

Let $\pi_1,\pi_2\in \irr$ be two ladder representations. By \cite[Theorem 1.2]{me-decomp}, there is a finite set $\mathcal{B}(\pi_1,\pi_2)\subset \irr$, for which we can write
\[
[\pi_1\times \pi_2] = \sum_{\sigma \in \mathcal{B}(\pi_1,\pi_2)} [\sigma] \in \mathcal{R}\;.
\]
In other words, the product is multiplicity-free.

Let us also consider the set
\[
\mathcal{C}(\pi_1,\pi_2) = \left\{\sigma \in \irr\;:\; m(\sigma_\otimes, \mathbf{r}_{\alpha_\sigma}(\pi_1\times \pi_2))>0 \right\}\;.
\]
Because of inequality (\ref{maj}) in Section \ref{sect-ind}, we clearly have the inclusion
\[
\mathcal{B}(\pi_1,\pi_2) \subset \mathcal{C}(\pi_1,\pi_2)\;.
\]
Now, let us write $\pi_1 = L(\gotM^{e,\rho}_{\lambda_1,\mu_1})$ and $\pi_2 = L(\gotM^{e,\rho}_{\lambda_2,\mu_2})$\footnote{ We will assume that $\pi_1,\pi_2\in \irr_{\langle \rho \rangle}$, for some $\rho\in \mathcal{C}$. In the absence of such condition $\pi_1\times\pi_2$ is irreducible and our discussion becomes trivial.}. Let $\lambda\in \mathcal{P}_n$ (resp. $\mu\in \mathcal{P}_n$) be the tuple constructed by taking (in ascending order) the union of all entries in $\lambda_1$ and $\lambda_2$ (resp. $\mu_1$ and $\mu_2$). Then, $\gotM^{e,\rho}_{\lambda_1,\mu_1} + \gotM^{e,\rho}_{\lambda_2,\mu_2} = \gotM^{w,\rho}_{\lambda,\mu}$ for some $w\in S_n$.

We will say that a permutation $x\in S_n$ is \textit{$321$-avoiding} if there is no sequence $i_1< i_2 < i_3$ in $\{1,\ldots,n\}$ for which $x(i_1)>x(i_2) > x(i_3)$ holds.

Let us consider the set
\[
\mathcal{D}(\pi_1,\pi_2) = \left\{\sigma = L(\gotM^{x,\rho}_{\lambda,\mu})\in \irr \;:\right.
\]
\[
\left.\mbox{:\;The longest permutation }x'\in S_n\mbox{ for which }\gotM^{x',\rho}_{\lambda,\mu} = \gotM^{x,\rho}_{\lambda,\mu}\mbox{ holds is 321-avoiding.}\right\}\;.
\]
It was shown in \cite[Corollary 4.13]{me-decomp} that
\[
\mathcal{B}(\pi_1,\pi_2)\subset \mathcal{D}(\pi_1,\pi_2)\;.
\]
Our main goal is the proof of the following equality which was conjectured in \cite[Conjecture 7.1]{me-decomp}.
\begin{theorem}\label{thm-main}
  \[
  \mathcal{B}(\pi_1,\pi_2) = \mathcal{C}(\pi_1,\pi_2) \cap \mathcal{D}(\pi_1,\pi_2)\;.
  \]
\end{theorem}

The theorem has the following corollary (which in fact was also shown to imply Theorem \ref{thm-main} in \cite[Proposition 7.3]{me-decomp}) that has its own interest.

\begin{corollary}[``Catalan conjecture" - Erez Lapid]\label{cor-lapid}
Suppose that $\pi_1,\pi_2$ are as above, with $\lambda_1 = (\lambda_1^i)_{i=1}^{m_1}$, $\lambda_2 = (\lambda_2^i)_{i=1}^{m_2}$, $\mu_1 = (\mu_1^i)_{i=1}^{m_1}$, $\mu_2 = (\mu_2^i)_{i=1}^{m_2}$, so that $m_1-m_2 \in \{0,1\}$ and the inequalities
\[
\left\{\begin{array}{l}\lambda_1^1 < \lambda_1^2 < \lambda_2^1 < \lambda_2^2 < \ldots < \lambda_{m_2}^2\;(< \lambda_{m_2+1}^1 ) \\

\mu_1^1 < \mu_1^2 < \mu_2^1 < \mu_2^2 < \ldots < \mu_{m_2}^2\;(< \mu_{m_2+1}^1 ) \end{array}\right.
\]
hold. Suppose further that $Q(\lambda,\mu) = S_{n}$.

Then,
  \[
  \mathcal{B}(\pi_1,\pi_2) = \mathcal{D}(\pi_1,\pi_2)\;,
  \]
and the size $|\mathcal{B}(\pi_1,\pi_2)|$ is the $n$-th Catalan number.
\end{corollary}
\begin{proof}
It was shown in the proof of \cite[Proposition 7.3]{me-decomp} that in this case we have $\mathcal{D}(\pi_1,\pi_2)\subset \mathcal{C}(\pi_1,\pi_2)$. As for the cardinality statement, it can be easily seen that since $\pi_1,\pi_2$ is a regular pair of ladder representations, the set $ \mathcal{D}(\pi_1,\pi_2)$ is in a natural bijection with the set of $321$-avoiding permutations in $S_n$.
\end{proof}

Let us reduce Theorem \ref{thm-main} into a somewhat more approachable condition.

\begin{theorem}\label{thm-main2}
For all regular $\lambda,\mu\in \mathcal{P}_n$ with $\lambda_n\leq \mu_1$ and all $321$-avoiding $w,w'\in S_n$ for which $w\neq w'$, we have
\[
m(\sigma_\otimes, \mathbf{r}_{\alpha_\sigma}(\sigma'))=0\;,
\]
where $\sigma = L(\gotM^w_{\lambda,\mu})$ and $\sigma' = L(\gotM^{w'}_{\lambda,\mu})$.

\end{theorem}

\begin{proposition}\label{prop-redu}
Theorem \ref{thm-main2} implies Theorem \ref{thm-main}.
\end{proposition}
\begin{proof}
Recall again that by \cite[Proposition 7.3]{me-decomp} Theorem \ref{thm-main} is equivalent to Corollary \ref{cor-lapid}. In fact, it suffices to show that the corollary holds for pairs $\pi_1,\pi_2$ as in the statement that are also regular.

Let $\sigma\in \mathcal{D}(\pi_1,\pi_2)\subset \mathcal{C}(\pi_1,\pi_2)$ be given. From exactness of $\mathbf{r}_{\alpha_\sigma}$, there exists $\sigma'\in \mathcal{B}(\pi_1,\pi_2)$ with $m(\sigma_\otimes, \mathbf{r}_{\alpha_\sigma}(\sigma'))>0$. Since $\sigma, \sigma'$ are both in $\mathcal{D}(\pi_1,\pi_2)$, by Theorem \ref{thm-main2} we must have $\sigma\cong \sigma'$. Thus, $\sigma\in \mathcal{B}(\pi_1,\pi_2)$.

\end{proof}

We refer the reader to Section 8 of \cite{me-decomp} for a discussion of the interaction between the $321$-avoidance property and known smoothness properties of the associated Schubert varieties. In particular, special cases of Theorem \ref{thm-main2} were proved in the presence of an additional smoothness assumption.

\section{Quantum Invariants}\label{sect-q}
\subsection{Lemmas on quantum shuffles}\label{sect-lemmas}
It is easily seen from the definition of the quantum shuffle product on $\mathcal{F}_{\mathbb{Q}(q)}$, that for any given words $w,z,x\in \mathcal{M}$ with $D^{\mathbb{Q}}_x(w\shuffle z)=1$, there is an integer $d(w,z;\;x)$ such that $D^{\mathbb{Q}(q)}_x(w \ast z) = q^{-d(w,z;\;x)}$.

\begin{lemma}\label{lem-ind-words}

Suppose that $w,w',w'',w_1,w'_1, w''_1, w_2, w'_2,w''_2\in \mathcal{M}$ are words such that $w = w'w''$, $w_1 = w'_1w''_1$, $w_2 = w'_2w''_2$ and that
\[
D_{w}(w_1\shuffle w_2) = D_{w'}(w'_1\shuffle w'_2)= D_{w''}(w''_1\shuffle w''_2) =1\;.
\]

Then,
\[
d(w_1,w_2;\; w) = d(w'_1,w'_2;\; w') + d(w''_1,w''_2;\; w'') + (|w'_2|,|w''_1| )\;.
\]

\end{lemma}

\begin{proof}

We prove by induction on the length of $w''$. If $w''$ is the empty word, there is nothing to prove. Otherwise, we write $w'' = \overline{w''}a$ for some $\overline{w''}\in \mathcal{M}$ and $a\in \mathcal{I}$. Similarly, $w = \overline{w}a$, where $\overline{w} = w' \overline{w''}$.

Suppose first that $w''_2 = \overline{w''_2}a$ for some $\overline{w''_2}\in \mathcal{M}$. Then, $D_{\overline{w''}} (w''_1\shuffle \overline{w''_2})=1$ and $w_2 = \overline{w_2}a$, where $\overline{w_2} = w'_2\overline{w''_2}$. Similarly, we see that $D_{\overline{w}}(w_1\shuffle \overline{w_2})=1$. By the induction hypothesis
\[
d(w_1,\overline{w_2};\;\overline{w}) = d(w'_1,w'_2;\; w') + d(w''_1,\overline{w''_2};\; \overline{w''}) + (|w'_2|,|w''_1| )\;.
\]
Yet, by the definition of the $\ast$ product we know that $d(w_1,\overline{w_2};\;\overline{w}) = d(w_1,w_2;\;w)$ and $d(w''_1,\overline{w''_2};\; \overline{w''}) = d(w''_1,w''_2;\; w'')$. Thus we are finished.

Otherwise, we must have $w''_1 = \overline{w''_1}a$ for a certain $\overline{w''_1}\in \mathcal{M}$. Then, $D_{\overline{w''}} (\overline{w''_1}\shuffle w''_2)=1$ and $w_1 = \overline{w_1}a$, where $\overline{w_1} = w'_1\overline{w''_1}$. Again, we have $D_{\overline{w}}(\overline{w_1}\shuffle w_2)=1$ and by the induction hypothesis
\[
d(\overline{w_1},w_2;\;\overline{w}) = d(w'_1,w'_2;\; w') + d(\overline{w''_1},w''_2;\; \overline{w''}) + (|w'_2|,|\overline{w''_1}|)\;.
\]
The formula defining $\ast$ shows that $d(w_1,w_2;\;w) = d(\overline{w_1},w_2;\;\overline{w})  + (|w_2|,a)$ and $d(w''_1,w''_2;\; w'') = d(\overline{w''_1},w''_2;\; \overline{w''}) + (|w''_2|,a)$. Thus,
\[
d(w_1,w_2;\;w) = d(w'_1,w'_2;\; w') + d(w''_1,w''_2;\; w'') - (|w''_2|,a) + (|w'_2|,|\overline{w''_1}| ) + (|w_2|,a)\;,
\]
We are finished, since $-(|w''_2|,a) + (|w'_2|,|\overline{w''_1}| ) + (|w_2|,a) =  (|w'_2|, |w''_1|)$.

\end{proof}

\begin{lemma}\label{lem-words}
Given $a,b,c,d\in \{1,\ldots,r\}$ with $a\leq b$ and $a< c\leq d$, we have
\[
(|\epsilon(a,b)|,|\epsilon(c,d)|) = \left\{\begin{array}{rc} 1 & b=d \\ -1 & b=c-1 \\ 0 & \mbox{otherwise} \end{array} \right.
\]
In addition, when $b=c-1$, we have
\[
d(\epsilon(a,b), \epsilon(c,d);\; \epsilon(a,d)) = -1,\; d(\epsilon(c,d), \epsilon(a,b);\; \epsilon(a,d)) = 0\;.
\]

\end{lemma}

The next lemma follows from a simple computation using Lemma \ref{lem-words} and Lemma \ref{lem-ind-words}.

\begin{lemma}\label{lem-tech-words}
Let $a_1< \ldots< a_{2k}$ be integers. Define the words
\[
w_1 = \epsilon(a_{2k-2},a_{2k-1}-1)\epsilon(a_{2k-4},a_{2k-3}-1)\cdots\epsilon(a_2,a_3-1)\;,
\]
\[
w_2 =  \epsilon(a_{2k-1},a_{2k}-1)\epsilon(a_{2k-3},a_{2k-2}-1)\cdots\epsilon(a_1,a_2-1)
\]
in $\mathcal{M}$.

Then, $D^{\mathbb{Q}}_{\epsilon(a_1,a_{2k}-1)}(w_1\shuffle w_2)=1$ holds and
\[
d(w_1, w_2;\; \epsilon(a_1,a_{2k}-1)) =d(w_2, w_1  ;\; \epsilon(a_1,a_{2k}-1)) = 1-k\;.
\]

\end{lemma}

\subsection{Quantum multiplicities}\label{sec-multi}

Recall that to a pair of ladder representations $\pi_1,\pi_2\in \irr^r$ corresponds a pair of dual canonical basis elements $b_1= b(\pi_1),b_2= b(\pi_2)\in \mathcal{B}$.

We can write a finite sum
\[
b_2b_1 = \sum_{b\in \mathcal{B}} c^b_{b_2,b_1} b\in U_\mathcal{A}\;,
\]
for some $0\neq c^b_{b_2,b_1}\in \mathcal{A}$. By dualizing \cite[Theorem 14.4.13(b)]{lusztig-book}, we see that $ c^b_{b_2,b_1} \in \mathbb{N}[q,q^{-1}]\subset \mathcal{A}$ for all $b\in \mathcal{B}$. When applying the specialization morphism $S$, we obtain the new equation
\[
[\pi_1][\pi_2] = [\pi_2][\pi_1] = \sum_{b\in \mathcal{B}(b_1,b_2)}  c^b_{b_2,b_1}(1)S(b)
\]
in $\mathcal{R}^r$, where $\mathcal{B}(b_1,b_2)=\mathcal{B}(b_2,b_1) = \{b\in \mathcal{B}\,:\, S(b)\in \mathcal{B}(\pi_1,\pi_2)\}$. In particular, $c_b^{b_2,b_1}=0$ for all $b\not\in \mathcal{B}(b_1,b_2)$.

Yet, from \cite[Theorem 1.2]{me-decomp} we know that the coefficients in the above equation must all be equal to $1$. Thus,
\[
b_2b_1 = \sum_{b\in \mathcal{B}(b_1,b_2)} q^{-d(b_1,b_2;\; b)} b
\]
holds, for some (possibly negative) integers $d(b_1,b_2;\; b)$.

Note, that the order of $b_1,b_2$ is reversed on purpose. Although such notation may look unnatural at first, the product-reversing property of $\Phi$ will make it better suited for our needs.

It will also be useful to write $d(\pi_1,\pi_2;\; \sigma) := d(\gotM_1,\gotM_2;\; \gotN) := d(b_1,b_2;\; b)$, where $\pi_1 = L(\gotM_1),\, \pi_2=L(\gotM_2), \sigma = L(\gotN) = S(b)$, for any $b\in \mathcal{B}(b_1,b_2)$.

\begin{proposition}\label{prop-mult}
Suppose that $b_1,b_2\in \mathcal{B}$ are such that $S(b_1),\,S(b_2)$ are ladder representations.

Then for every $\sigma = L(\gotM^x_{\lambda,\mu})\in \irr^r$ with regular $\lambda$, there is at most one element $b\in \mathcal{B}(b_1,b_2)$ for which $D^{\mathbb{Q}(q)}_{\epsilon(\sigma)}(\Phi(b)) = 1$. For all $b'\in \mathcal{B}\setminus \{b\}$, $D^{\mathbb{Q}(q)}_{\epsilon(\sigma)}(\Phi(b'))=0$.

As a consequence, $ D^{\mathbb{Q}(q)}_{\epsilon(\sigma)}(\Phi(b_2b_1)) = q^{-d(b_1,b_2;\; b)}$ if such $b$ exists, and $ D^{\mathbb{Q}(q)}_{\epsilon(\sigma)}(\Phi(b_1b_2)) = 0$ otherwise.
\end{proposition}

\begin{proof}
By Lemma \ref{lemma-first} and \cite[Proposition 5.3]{me-decomp}, we have
\[
D^{\mathbb{Q}(q)}_{\epsilon(\sigma)}(\Phi(b_2b_1))(1) = D^{\mathbb{Q}}_{\epsilon(\sigma)}(\ch([\pi_1][\pi_2])) \in \{0,1\}\;.
\]
On the other hand,
\[
D^{\mathbb{Q}(q)}_{\epsilon(\sigma)}(\Phi(b_2b_1)) =\sum_{b\in \mathcal{B}(b_1,b_2)}  q^{-d(b_1,b_2;\; b)}  D^{\mathbb{Q}(q)}_{\epsilon(\sigma)}(\Phi(b))\;.
\]

It is easily deduced from the above two equations and Proposition \ref{leclerc-thm} that there is at most one element $b\in \mathcal{B}(b_1,b_2)$ such that $D^{\mathbb{Q}(q)}_{\epsilon(\sigma)}(\Phi(b)) = q^c$ for some power $c$, and that for all $b'\in \mathcal{B}\setminus \{b\}$, $D^{\mathbb{Q}(q)}_{\epsilon(\sigma)}(\Phi(b'))=0$.

Recalling again Proposition \ref{leclerc-thm}, the condition $D^{\mathbb{Q}(q)}_{\epsilon(\sigma)}(\Phi(b))(q) = D^{\mathbb{Q}(q)}_{\epsilon(\sigma)}(\Phi(b))(q^{-1})$ implies $c=0$.

\end{proof}

Suppose now that $\sigma = L(\gotN)\in \mathcal{C}(\pi_1,\pi_2)$, and that $\pi_1,\pi_2$ form a regular pair. From the above proposition (regularity condition on $\gotN$ follows from Proposition \ref{prop-matrix}), we see that there exists an integer
\[
d_\otimes (\pi_1,\pi_2;\; \sigma)= d_\otimes(\gotM_1,\gotM_2;\;\gotN) = d_\otimes (b_1,b_2;\; b(\sigma))\;,
\]
for which we can write
\[
D^{\mathbb{Q}(q)}_{\epsilon(\sigma)}(\Phi(b_2b_1))=q^{-d_\otimes (\pi_1,\pi_2;\; \sigma)}\;.
\]
In fact, we know that $d_\otimes (b_1,b_2;\; b(\sigma)) = d(b_1,b_2;\; b)$ for a certain $b\in \mathcal{B}(b_1,b_2)$.

Note, that in general, we may have $b\neq b(\sigma)$. Of course, when $\sigma\in  \mathcal{B}(\pi_1,\pi_2)$, that is, $b(\sigma)\in  \mathcal{B}(b_1,b_2)$, it is clear from definitions that $b=b(\sigma)$.

\subsection{Combinatorial invariants}\label{sect-mult}
Let us fix a regular pair of ladder representations $\pi_1 = L(\gotM_1),\pi_2=L(\gotM_2)\in \irr^r$ for the rest of this section. Let us take record of the combinatorial data $(w,J_1,J_2)$ of the pair $\pi_1,\pi_2$, so that $w\in S_n$, $J_1 = \{i_1< \ldots < i_{|J_1|}\}$ and $J_2 = \{i'_1 <\ldots < i'_{|J_2|}\}$.

We also write $b_1,b_2\in\mathcal{B}$ and $\gotM_1,\gotM_2\in \mathbb{N}(\seg_0)$ for the corresponding basis elements and multisegments, as before.

We would like to attain a combinatorial formula for the number $d_\otimes (\pi_1,\pi_2;\; \sigma)$, for each given $\sigma\in \mathcal{C}(\pi_1,\pi_2)$.

Let us fix such $\sigma = L(\gotN)\in \mathcal{C}(\pi_1,\pi_2)$. We deduce from Lemma \ref{lemma-first}, the fact that $ch$ is a homomorphism into the shuffle algebra and \cite[Proposition 5.3]{me-decomp}, that $D^{\mathbb{Q}}_{\epsilon(\sigma)}(\ch(\pi_1)\shuffle \ch(\pi_2))=1$. Since the $\mathcal{M}$-coefficients in the expansion of $\ch(\pi_i)$ are clearly positive, we see that there must be unique words $\epsilon_1,\epsilon_2\in \mathcal{M}$, so that
\[
D^{\mathbb{Q}}_{\epsilon_i}(\ch(\pi_i)) =1\;,\; i=1,2\;,
\]
and that $D^{\mathbb{Q}}_{\epsilon(\sigma)}(\epsilon_1\shuffle \epsilon_2) =1$.

By \cite[Corollary 48]{leclerc-shuffle}, $D^{\mathbb{Q}(q)}_w(\Phi(b_i))\in\{0,1\}$ for all $w\in \mathcal{M}$. Hence,

\[
D^{\mathbb{Q}(q)}_{\epsilon(\sigma)}(\Phi(b_2b_1))= D^{\mathbb{Q}(q)}_{\epsilon(\sigma)}(\Phi(b_1)\ast\Phi(b_2)) =
\]
\[
=  q^{-d(\epsilon_1,\epsilon_2;\;\epsilon(\sigma))} D^{\mathbb{Q}(q)}_{\epsilon_1}(\Phi(b_1))D^{\mathbb{Q}(q)}_{\epsilon_2}(\Phi(b_2)) = q^{-d(\epsilon_1,\epsilon_2;\;\epsilon(\sigma))}\;,
\]
In other words, $d_\otimes(\pi_1,\pi_2;\;\sigma) = d(\epsilon_1,\epsilon_2;\;\epsilon(\sigma))$.

In fact, the words $\epsilon_1,\epsilon_2$ can be extracted explicitly using the information given in Proposition \ref{prop-matrix}. Recall, that the mentioned proposition supplies us with a matrix of integers
\[
C = C(\pi_1,\pi_2;\;\sigma)= C(\gotM_1,\gotM_2;\;\gotN)=(c_i^j)_{i=1,\ldots, n}^{j=1,\ldots, n+1}\;,
\]
together with the irreducible representations $\sigma_1= \otimes_j\sigma_1^j,\,\sigma_2= \otimes_j\sigma_2^j$, for which $m(\sigma_\otimes, \sigma_1\times \sigma_2)=1$.

Since $D^{\mathbb{Q}}_{\epsilon(\sigma)}(\ch_{\alpha_\sigma}(\sigma_\otimes))=1$, we know that
\[
D^{\mathbb{Q}}_{\epsilon(\sigma)}((\ch(\sigma^n_1)\shuffle \ch(\sigma^n_2))(\ch(\sigma^{n-1}_1)\shuffle \ch(\sigma^{n-1}_2))\cdots (\ch(\sigma^1_1)\shuffle \ch(\sigma^1_2))) = 1\;.
\]
It also follows that $D^{\mathbb{Q}}_{\epsilon_i}(\ch_{\alpha_i}(\sigma_i)) = 1$, for $i=1,2$ (and suitable $\alpha_i$). Subsequently, it can be easily deduced from the given form of $\sigma_1,\sigma_2$ in Proposition \ref{prop-matrix} that
\[
\epsilon_1 = \epsilon_1^n\epsilon^{n-1}_1\cdots \epsilon^1_1,\quad \epsilon_2 = \epsilon^n_2\epsilon^{n-1}_2\cdots \epsilon^1_2\;,
\]
where $\epsilon_1^j = \epsilon(c^j_{i_{|J_1|}}, c^{j+1}_{i_{|J_1|}}-1)\cdots \epsilon(c^j_{i_{1}}, c^{j+1}_{i_{1}}-1)$ and $\epsilon_2^j = \epsilon(c^j_{i'_{|J_2|}}, c^{j+1}_{i'_{|J_2|}}-1)\cdots \epsilon(c^j_{i'_{1}}, c^{j+1}_{i'_{1}}-1)$.

We are now ready to express some quantum multiplicities in terms of the following statistics on the matrix $C$:
\[
\alpha_{J_1}(\gotM_1,\gotM_2;\;\gotN) := \#\{(i,j)\;:\; i\in J_1,\,1\leq j\leq n,\,c^j_i< c^{j+1}_i\} - |J_1|\;,
\]
\[
\alpha_{J_2}(\gotM_1,\gotM_2;\;\gotN) := \#\{(i,j)\;:\; i\in J_2,\,1\leq j\leq n,\,c^j_i< c^{j+1}_i\} - |J_2|\;.
\]

Note, that $\alpha_{J_1},\alpha_{J_2}$ are always non-negative. This is because $c^1_i=\lambda_i  < \mu_{w(i)}+1 = c^{n+1}_i$, for all $1\leq i\leq n$.

\begin{proposition}\label{prop-comb}
For $\gotM_1,\gotM_2,\gotN$ as above,
\[
d_\otimes (\gotM_1,\gotM_2;\;\gotN) = \alpha_{J_1}(\gotM_1,\gotM_2;\;\gotN) - \alpha_{J_2}(\gotM_1,\gotM_2;\;\gotN)\;.
\]
\end{proposition}

\begin{proof}

We give a proof by induction on $n$.

Let us write $\gotN = \gotM^x_{\lambda,\mu}$. Let us consider the truncated matrix $C' = (c_i^j)_{i=2,\ldots, n}^{j=1,\ldots, n}$. It is easy to verify that $C' = C(\gotM'_1, \gotM'_2;\; \gotN')$, where $\gotN' = \sum_{j=2}^{n} [\lambda_j, \mu_{x(j)}]$, $\gotM'_1 = \sum_{i\in J_1\setminus \{1\} } [\lambda_i, c_i^n -1]$ and $\gotM'_2 = \sum_{i\in J_2\setminus \{1\} } [\lambda_i, c_i^n -1]$.

Let $\epsilon'_1,\epsilon'_2\in \mathcal{M}$ be the words for which $ D^{\mathbb{Q}}_{\epsilon'_i}(\ch(L(\gotM'_i))) =1,\;i=1,2$, and $D^{\mathbb{Q}}_{\epsilon(\gotN')}(\epsilon'_1 \shuffle \epsilon'_2) =1$. Clearly, $\epsilon_1 = \epsilon^n_1 \epsilon'_1$, $\epsilon_2 = \epsilon^n_2 \epsilon'_2$ and $\epsilon(\gotN) = \epsilon(\lambda_1, \mu_{x(1)})\epsilon(\gotN')$.


We write $\beta_{J_l} = \#\{i\in J_l\;:\; c^n_i<c^{n+1}_i\}$, $l=1,2$. Note, that $\beta_{J_l}$ stands for the number of words in the product which defines $\epsilon^n_l$.

Let $i_0$ be the index for which $c^n_{i_0}< c^{n+1}_{i_0}$ and $c^n_{i_0}$ is maximal. We also set the indicators
\[
\delta_b = \left\{ \begin{array}{cc} 1  & 1\in J_2 \\ 0 & 1\in J_1 \end{array}\right. ,\quad \delta_e = \left\{ \begin{array}{cc} 1  & i_0\in J_2 \\ 0 & i_0\in J_1 \end{array}\right. .
\]

The condition $D_{\epsilon(\lambda_1, \mu_{x(1)})} (\epsilon^n_1, \epsilon^n_2)=1$ is clear from condition (\ref{cond-three}) in Proposition \ref{prop-matrix}. It is also clear that the equality $\beta_{J_1} = \beta_{J_2} +1 -\delta_b -\delta_e$ holds.

In case $\delta_b= \delta_e$, we see by Lemma \ref{lem-tech-words} that $d( \epsilon^n_1 , \epsilon^n_2;\; \epsilon(\lambda_1, \mu_{w(1)})) = -\min\{\beta_{J_1},\beta_{J_2}\}$.

Otherwise, when $\delta_b=1$ and $\delta_e=0$, we have $\epsilon^n_1= \epsilon(c^n_{i_0}, c^{n+1}_{i_0}-1)\epsilon''$ for $\epsilon''\in \mathcal{M}$ and by Lemma \ref{lem-tech-words}

\[
d( \epsilon^n_1 , \epsilon^n_2;\; \epsilon(\lambda_1, \mu_{x(1)})) = d( \epsilon'' , \epsilon^n_2;\; \epsilon(\lambda_1, c^{n}_{i_0}-1)) = 1-\beta_{J_2}\;.
\]
In the last case of $\delta_b=0$ and $\delta_e=1$, we have $\epsilon^n_2= \epsilon(c^n_{i_0}, c^{n+1}_{i_0}-1)\epsilon''$ for $\epsilon''\in \mathcal{M}$ and by applying again the lemmas of Section \ref{sect-lemmas}, we obtain
\[
d( \epsilon^n_1 , \epsilon^n_2;\; \epsilon(\lambda_1, \mu_{x(1)})) = -1 + d( \epsilon^n_1, \epsilon'' ;\; \epsilon(\lambda_1, c^{n}_{i_0}-1)) = -\beta_{J_2}\;.
\]
Altogether, we see that the formula $d( \epsilon^n_1 , \epsilon^n_2;\; \epsilon(\lambda_1, \mu_{x(1)})) =- \beta_{J_2} + \delta_b$ holds in all cases.

Note, that $\alpha_{J_1}(\gotM_1,\gotM_2;\;\gotN)  = \alpha_{J_1}(\gotM'_1,\gotM'_2;\;\gotN') + \beta_{J_1} +\delta_b-1$, and that $\alpha_{J_2}(\gotM_1,\gotM_2;\;\gotN)  = \alpha_{J_2}(\gotM'_1,\gotM'_2;\;\gotN') + \beta_{J_2}-\delta_b$.

By the induction hypothesis and Lemma \ref{lem-ind-words}, we have
\[
d(\gotM_1,\gotM_2;\;\gotN) = d(\epsilon'_1, \epsilon'_2 , \epsilon(\gotN')) + d(  \epsilon^n_1 , \epsilon^n_2;\; \epsilon(\lambda_1, \mu_{w(1)})) + (|\epsilon^n_2|, |\epsilon'_1|)=
\]
\[
= \alpha_{J_1}(\gotM'_1,\gotM'_2;\;\gotN') - \alpha_{J_2}(\gotM'_1,\gotM'_2;\;\gotN') - \beta_{J_2} +\delta_b + (|\epsilon^n_2|, |\epsilon'_1|) =
\]
\[
= \alpha_{J_1}(\gotM_1,\gotM_2;\;\gotN) - \alpha_{J_2}(\gotM_1,\gotM_2;\;\gotN) - \beta_{J_1} +1 - \delta_b +  (|\epsilon^n_2|, |\epsilon'_1|)\;.
\]
We finish by proving the following claim.
\begin{claim}
\[
 (|\epsilon^n_2|, |\epsilon'_1|) = \beta_{J_2} - \delta_e = \beta_{J_1} +\delta_b -1\;.
\]
\end{claim}
Since $|\epsilon'_1| = \sum_{i\in J_1\setminus \{1\} } |\epsilon(\lambda_i, c_i^n -1)|$, we can write
\[
(|\epsilon^n_2|, |\epsilon'_1|) = \sum_{j\in J_2,\, i\in J_1\setminus \{1\}} ( |\epsilon(c^n_j, c^{n+1}_j-1)|,|\epsilon(\lambda_i, c_i^n -1)| ) =
\]
\[
= \sum_{\substack{ j\in J_2:\, c^n_j < c^{n+1}_j\\ i\in J_1\setminus \{1\}}} ( |\epsilon(c^n_j, c^{n+1}_j-1)|,|\epsilon(\lambda_i, c_i^n -1)| ) \;.
\]
According to Lemma \ref{lem-words}, the summands in the above equation can only be non-zero for the cases that $c^n_i\in \{c^n_j, c^{n+1}_j\}$, for some $i\in J_1$ and $j\in J_2$.

We will first show that $c^n_i\neq c^n_j$ for all $i\in J_1$ and $j\in J_2$ such that $c^n_j < c^{n+1}_j$. Indeed, suppose that $c^n_{i_1}=c^n_{j_1} < c^{n+1}_{j_1}$ for some $i_1\in J_1$ and $j_1\in J_2$. Since the situation of $c^n_{i_1}=c^n_{j_1}=\lambda_1$ is impossible (it would have implied $\lambda_{i_1} = \lambda_{j_1} = \lambda_1$), the properties of $C$ imply that there is an index $i_2\in J_1$ for which $c^{n+1}_{i_2} = c^n_{j_1}$. In particular, $c^{n+1}_{i_1}\neq c^{n+1}_{i_2} = c^n_{i_1}$. This means that $[c^n_{i_1}, c^{n+1}_{i_1}]$ is a non-empty segment, which intersects $[c^n_{j_1}, c^{n+1}_{j_1}]$. This is a contradiction to the properties of $C$.

Thus, Lemma \ref{lem-words} points that
\[
(|\epsilon'_2|, |\epsilon^n_1|) = \#\{ j\in J_2\;:\;  c^n_j < c^{n+1}_j,\; \exists i\in J_1 \setminus \{1\},\, c^n_i = c^{n+1}_j \}\;.
\]
Yet, the properties of $C$ imply that for every $j\in J_2$ with $c^n_j < c^{n+1}_j$ such that $j\neq i_0$, there must exist $i\in J_1$ (and clearly $i\neq 1$) such that $c^{n+1}_j = c^n_i < c^{n+1}_i$.

\end{proof}

\begin{example}\label{examp-2}

Consider the multisegments $\gotM_1 = [0,6] + [3,8] + [4,9]$ and $\gotM_2 = [1,5] + [2,7]$. The resulting combinatorial data is given as $((21345), \{1,4,5\}, \{2,3\})$. Note, that for $x_1 = (24513)$ we have $L(\gotM^x_{\lambda,\mu})\in \mathcal{C}(\gotM_1,\gotM_2)$ because of the existence of the matrix
\[
C(\gotM_1,\gotM_2;\;\gotM_{\lambda,\mu}^{x_1}) = \left(\begin{array}{cccccc} 0 & 0 & 0 & 0 & 0 & 7  \\
                                                                        1 & 1 & 1 & 1 & 6 & 6  \\
                                                                        2 & 2 & 2 & 8 & 8 & 8  \\
                                                                        3 & 3 & 6 & 6 & 9 & 9  \\
                                                                        4 & 8 & 8 & 10 & 10 & 10 \end{array}\right)\;.
\]
(Here, for $(c_i^j)$ we write $i$ as the row index and $j$ as the column index.)

Consequently, we see that $\alpha_{J_1}(\gotM_1,\gotM_2;\;\gotM_{\lambda,\mu}^{x_1}) = 2$ and $\alpha_{J_2}(\gotM_1,\gotM_2;\;\gotM_{\lambda,\mu}^{x_1}) =0$.

Similarly, for $x_2 = (34512)$, we have
\[
C(\gotM_1,\gotM_2;\;\gotM_{\lambda,\mu}^{x_2}) = \left(\begin{array}{cccccc} 0 & 0 & 0 & 0 & 0 & 7  \\
                                                                        1 & 1 & 1 & 1 & 6 & 6  \\
                                                                        2 & 2 & 2 & 7 & 7 & 8  \\
                                                                        3 & 3 & 6 & 6 & 9 & 9  \\
                                                                        4 & 7 & 7 & 10 & 10 & 10 \end{array}\right)\;,
\]
with $\alpha_{J_1}(\gotM_1,\gotM_2;\;\gotM_{\lambda,\mu}^{x_2}) = 2$ and $\alpha_{J_2}(\gotM_1,\gotM_2;\;\gotM_{\lambda,\mu}^{x_2}) =1$.

Similarly, for $x_3 = (41235)$, we have
\[
C(\gotM_1,\gotM_2;\;\gotM_{\lambda,\mu}^{x_3}) = \left(\begin{array}{cccccc} 0 & 0 & 0 & 0 & 0 & 7  \\
                                                                        1 & 1 & 1 & 1 & 6 & 6  \\
                                                                        2 & 2 & 2 & 7 & 7 & 8  \\
                                                                        3 & 3 & 8 & 8 & 8 & 9  \\
                                                                        4 & 10 & 10 & 10 & 10 & 10 \end{array}\right)\;,
\]
with $\alpha_{J_1}(\gotM_1,\gotM_2;\;\gotM_{\lambda,\mu}^{x_3}) = 1$ and $\alpha_{J_2}(\gotM_1,\gotM_2;\;\gotM_{\lambda,\mu}^{x_3}) =1$.
\end{example}

The following lemma will be crucial towards the later study of the representations in $\mathcal{C}(\pi_1,\pi_2)\cap \mathcal{D}(\pi_1,\pi_2)$.

\begin{lemma}\label{lem-import}

Let $\sigma\in \mathcal{C}(\pi_1,\pi_2)\cap\mathcal{D}(\pi_1,\pi_2)$ be a representation with $\gotM_1,\gotM_2,\gotN\in \mathbb{N}(\seg_0)$ as above. Suppose that $\alpha_{J_2}(\gotM_1,\gotM_2;\;\gotN)>0$.

Then, there exists $\sigma'\in \mathcal{C}(\pi_1,\pi_2)$, so that $d_\otimes(\pi_1,\pi_2;\;\sigma)<d_\otimes(\pi_1,\pi_2;\;\sigma')$.

\end{lemma}

\begin{proof}
In order to produce such $\sigma' = L(\gotN')$, it is enough to construct a new matrix $\overline{C} = C(\gotM_1,\gotM_2;\;\gotN')$ which satisfies all conditions listed in the statement of Proposition \ref{prop-matrix}.

Let us write $\gotN = \gotM^x_{\lambda,\mu} = \sum_{i=1}^n \Delta_i$, so that $\Delta_i = [\lambda_{n+1-i}, \mu_{x(n+1-i)}]$ for all $1\leq i\leq n$.

We will produce $\overline{C}$ and $\sigma'$ out of adjustments on the matrix $C = C(\gotM_1,\gotM_2;\;\gotN)=(c_i^j)_{i=1,\ldots, n}^{j=1,\ldots, n+1}$, provided again by Proposition \ref{prop-matrix}.

Let $i_0\in J_2$ be the maximal index for which
\[
\#\{1\leq j\leq n\;:\;c^j_{i_0}< c^{j+1}_{i_0}\}>1\;.
\]
From the assumption on $\alpha_{J_2}$, such $i_0$ exists. Let $j_0$ be the maximal index for which $c^{j_0}_{i_0}< c^{j_0+1}_{i_0}$. Let $j_1<j_0$ be the maximal index for which $c^{j_1}_{i_0}< c^{j_0}_{i_0}$.

The multisegment $\gotN'$, to which the matrix $\overline{C}$ will correspond, is in fact produced as $\gotN' = \gotM^{x'}_{\lambda,\mu}$, where $x' = x\cdot\tau$ with $\tau$ being a certain transposition involving $j_0$.

Since $\sigma\in \mathcal{D}(\pi_1,\pi_2)$, we know that $x$ is $321$-avoiding, which means there cannot be a sequence $1\leq i'_1< i'_2<i'_3\leq n$ for which $e(\Delta_{i'_1})<e(\Delta_{i'_2})<e(\Delta_{i'_3})$.

\begin{claim}\label{clm-first}
For all $i_0<i\in J_2$ and $j_1\leq j\leq n$, we have $c^j_i = c^{n+1}_i = \mu_{w(i)}+1$.
\end{claim}

Note, that $c^1_i = \lambda_i < c^{j_1+1}_{i_0}< c^{j_1}_i \leq  c^j_i$. The statement then follows from maximality of $i_0$.

\begin{claim}\label{clm-second}
For all $j_1\leq j\leq j_0$ and $i\in J_1$ for which $c^j_i < c^{j+1}_i$ and $c^{j_0}_{i_0}< c^{j+1}_i$ hold, we have $c^{j+1}_{i} -1 = e(\Delta_{j})$.
\end{claim}
Assume the contrary for some $(i,j)$. Since $[c^{j}_{i}, c^{j+1}_{i}-1]\subseteq \Delta_{j}$, we have $c^{j+1}_{i}\leq  e(\Delta_{j})$. Hence, there exists $i'\in J_2$ such that $c^{j+1}_{i}=c^{j}_{i'}< c^{j+1}_{i'}$.

Now, since $j\leq j_0$, we see that $c^j_{i_0}\leq c^{j_0}_{i_0}< c^j_{i'}$. Hence, $i_0<i'$. So, by claim \ref{clm-first} we know that $c^j_{i'} = c^{n+1}_{i'} \geq c^{j+1}_{i'}$, which gives a contradiction.
\\ \\
Recall that $[c^{j_0}_{i_0}, c^{j_0+1}_{i_0}-1]\subseteq \Delta_{j_0}$.

The rest of the proof will be split over three cases.
\\
\paragraph{Case I: $c^{j_0+1}_{i_0} \leq e(\Delta_{j_0})$} There exists $i_1\in J_1$ such that $c^{j_0+1}_{i_0}=c^{j_0}_{i_1}< c^{j_0+1}_{i_1}$. By claim \ref{clm-second} we know that $c^{j_0+1}_{i_1} -1 = e(\Delta_{j_0})$.

By Remark \ref{rmrk-matrix}, since $c^1_{i_0}\leq c^{j_1}_{i_0}< c^{j_0}_{i_0}$, we know that $c^{j_1+1}_{i_0} = c^{j_0}_{i_0} = \mu_{t}+1$ for a certain $1\leq t\leq n$. We write $s = n+1 - x^{-1}(t)$ to obtain $\mu_t = e(\Delta_{s})$. 

\begin{claim}\label{clm-third}
$c^{j_1}_{i_1} = c^{j_0}_{i_1}\;.$
\end{claim}
If the contrary holds, there exists an index $j_1\leq j < j_0$ for which $c^j_{i_1} < c^{j+1}_{i_1} = c^{j_0}_{i_1}$. From claim \ref{clm-second}, we see that  $c^{j+1}_{i_1} -1 = e(\Delta_{j})$.

Note, that the inequalities
\[
e(\Delta_s)+1 = c^{j_0}_{i_0} < c^{j_0+1}_{i_0} = c^{j_0}_{i_1} = e(\Delta_j)+1 < c^{j_0+1}_{i_1} = e(\Delta_{j_0})+1
\]
imply that $s<j$ and contradict the pattern avoidance condition on the triple $s<j<j_0$.

\begin{claim}\label{clm-fourth}
$e(\Delta_{j_1}) = c^{j_0}_{i_0} -1$
\end{claim}

Assume the contrary. Then $e(\Delta_{j_1}) \geq c^{j_0}_{i_0}$ and there exists $i_3\in J_1$ such that $c^{j_0}_{i_0}=c^{j_1}_{i_3}< c^{j_1+1}_{i_3}$. By claim \ref{clm-second}, $c^{j_1+1}_{i_3} -1 = e(\Delta_{j_1})$.

Our assumption also means that $s<j_1$. The pattern avoidance condition applied on the triple $s<j_1< j_0$ must now imply that $e(\Delta_{j_0})< e(\Delta_{j_1})$. In particular, $c^{j_0+1}_{i_1} < c^{j_1+1}_{i_3}$. Since $c^{j_1+1}_{i_1}\leq c^{j_0+1}_{i_1}$ we can deduce that $i_1<i_3$.

On the other hand, from Claim \ref{clm-third} we see the contradictory inequality
\[
c^{j_1}_{i_1} = c^{j_0}_{i_1} = c^{j_0+1}_{i_0}> c^{j_0}_{i_0} = c^{j_1}_{i_3}\;.
\]
\\

We are ready to construct the matrix $\overline{C}=(\overline{c}^j_i)$ for this case. We define
\[
\overline{c}^j_i = \left\{ \begin{array}{cc} c^{j_0+1}_{i_0} & i=i_0,\, j_1<j\leq j_0 \\ c^{j_0+1}_{i_1} & i=i_1,\, j_1<j\leq j_0  \\ c^j_i & \mbox{otherwise} \end{array} \right.\;.
\]
We will know that $\overline{C} = C(\gotM_1,\gotM_2;\;\gotN')$ for a certain multisegment $\gotN'$, if all conditions listed for such matrices are met. Most conditions can be easily deduced from what was claimed above. The non-trivial property we are left to check is the monotonicity of $\overline{C}$ along the lower index.

By Claim \ref{clm-first}, $c^j_i = c^{n+1}_i\geq c^{j_0+1}_i> c^{j_0+1}_{i_0}$, for all $i_0<i\in J_2$ and $j_1\leq j\leq j_0$. Hence, $\overline{c}^{j+1}_{i_0}< c^j_i$.

We also need to show that $\overline{c}^{j+1}_{i_1}< c^j_i$, for all $i_1<i\in J_1$ and $j_1\leq j\leq j_0$. Assume the contrary, i.e. $c^j_i\leq c^{j_0+1}_{i_1} $ for some $i_1<i\in J_1$ and $j_1\leq j\leq j_0$. Since $c^{j_0+1}_{i_1} < c^{j_0+1}_i$, by taking the maximal $j$ we can assume that $c^j_i<c^{j+1}_i$.

From Claim \ref{clm-third} we know that $c^j_i > c^j_{i_1} = c^{j_0}_{i_1} = c^{j_0+1}_{i_0} > c^1_{i_0}$. Hence, applying Remark \ref{rmrk-matrix} gives $b(\Delta_j)< c^j_i$. Thus, there exists $i'\in J_2$ such that $c^j_{i'} < c^{j+1}_{i'} = c^j_i$. By Claim \ref{clm-first} we must have $i'<i_0$. In particular,
\[
c^j_i = c^{j+1}_{i'} < c^{j+1}_{i_0} \leq c^{j_0+1}_{i_0} = c^{j_0}_{i_1} = c^j_{i_1}
\]
follows from Claim \ref{clm-third} again. This is a contradiction to $i_1<i$.
\\ \\
\paragraph{Case II: $c^{j_0+1}_{i_0} -1= e(\Delta_{j_0})$ and $e(\Delta_{j_1}) = c^{j_0}_{i_0} -1$.} Here we define $\overline{C}=(\overline{c}^j_i)$ as
\[
\overline{c}^j_i = \left\{ \begin{array}{cc} c^{j_0+1}_{i_0} & i=i_0,\, j_1<j\leq j_0 \\ c^j_i & \mbox{otherwise} \end{array} \right.\;.
\]
The proof that $\overline{C}$ satisfies all required conditions is now similar to Case I.
\\
In both cases I and II, the construction of $\overline{C}$ shows that $\alpha_{J_2}(\gotM_1,\gotM_2;\;\gotN') = \alpha_{J_2}(\gotM_1,\gotM_2;\;\gotN)-1$, while $\alpha_{J_1}(\gotM_1,\gotM_2;\;\gotN') =\alpha_{J_1}(\gotM_1,\gotM_2;\;\gotN)$. By Proposition \ref{prop-comb} we have finished.
\\ \\
\paragraph{Case III: $c^{j_0+1}_{i_0} -1= e(\Delta_{j_0})$ and $ c^{j_0}_{i_0} \leq e(\Delta_{j_1})$.} There exists $i_3\in J_1$ such that $c^{j_0}_{i_0}=c^{j_1}_{i_3}< c^{j_1+1}_{i_3}$. From Claim \ref{clm-second} we know that $c^{j_1+1}_{i_3} -1 = e(\Delta_{j_1})$.

In the same manner as in the proof of Claim \ref{clm-fourth} above, we can deduce from the pattern avoidance condition on $x$ that $e(\Delta_{j_0})<e(\Delta_{j_1})$, i.e. $c^{j_0+1}_{i_0}<c^{j_1+1}_{i_3}$. In particular, $c^{j_0+1}_{i_0}<c^{j_0+1}_{i_3}$.

So, we write $i_3\geq i_1\in J_1$ for the minimal index such that $c^{j_0+1}_{i_0}<c^{j_0+1}_{i_1}$. Since $e(\Delta_{j_0})+1<  c^{j_0+1}_{i_1}$, we see that $c^{j_0}_{i_1}= c^{j_0+1}_{i_1}$. Let us write $j_2\leq j_0$ for the minimal index such that $c^{j_2}_{i_1}= c^{j_0+1}_{i_1}$.

Since $c^{j_1}_{i_1} \leq c^{j_1}_{i_3} = c^{j_0}_{i_0} < c^{j_0+1}_{i_0} < c^{j_0+1}_{i_1}=c^{j_2}_{i_1}$ holds, we see that $j_1<j_2$. In particular, Claim \ref{clm-second} implies that $c^{j_2}_{i_1} -1 = e(\Delta_{j_2-1})$.

We can now define $\overline{C}=(\overline{c}^j_i)$ as
\[
\overline{c}^j_i = \left\{ \begin{array}{cc} c^{j_0+1}_{i_0} & i=i_1,\, j_2\leq j\leq j_0 \\ c^j_i & \mbox{otherwise} \end{array} \right.\;.
\]
Monotonicity along the lower index is preserved, since for every $i_1> i\in J_1$ and $j_2\leq j\leq j_0$, $c^{j+1}_i \leq c^{j_0+1}_{i} < c^{j_0+1}_{i_0}$ holds by minimality of $i_1$.

Thus, $\overline{C} =C(\gotM_1,\gotM_2;\;\gotN')$ with $\alpha_{J_2}(\gotM_1,\gotM_2;\;\gotN) = \alpha_{J_2}(\gotM_1,\gotM_2;\;\gotN)$ and $\alpha_{J_1}(\gotM_1,\gotM_2;\;\gotN) = \alpha_{J_1}(\gotM_1,\gotM_2;\;\gotN)+1$. The proof is finished for this case as well, by Proposition \ref{prop-comb}.
\end{proof}

\begin{example}
Let us observe the statement of Lemma \ref{lem-import} and its proof through Example \ref{examp-2}. Since $\alpha_{J_2}(\gotM_1,\gotM_2;\;\gotM_{\lambda,\mu}^{x_3})>0$, Case I in the proof suggests taking $x_4 = x_3\cdot (13) = (21435)$ and obtain
\[
C(\gotM_1,\gotM_2;\;\gotM_{\lambda,\mu}^{x_3}) = \left(\begin{array}{cccccc} 0 & 0 & 0 & 0 & 0 & 7  \\
                                                                        1 & 1 & 1 & 1 & 6 & 6  \\
                                                                        2 & 2 & 2 & 8 & 8 & 8  \\
                                                                        3 & 3 & 8 & 9 & 9 & 9  \\
                                                                        4 & 10 & 10 & 10 & 10 & 10 \end{array}\right)\;,
\]
with $\alpha_{J_1}(\gotM_1,\gotM_2;\;\gotM_{\lambda,\mu}^{x_4}) = 1$ and $\alpha_{J_2}(\gotM_1,\gotM_2;\;\gotM_{\lambda,\mu}^{x_4}) =0$.

Similarly, since $\alpha_{J_2}(\gotM_1,\gotM_2;\;\gotM_{\lambda,\mu}^{x_2})>0$, Case III in the proof suggests taking $x_5 = x_2\cdot (12) = (43512)$ and obtain
\[
C(\gotM_1,\gotM_2;\;\gotM_{\lambda,\mu}^{x_2}) = \left(\begin{array}{cccccc} 0 & 0 & 0 & 0 & 0 & 7  \\
                                                                        1 & 1 & 1 & 1 & 6 & 6  \\
                                                                        2 & 2 & 2 & 7 & 7 & 8  \\
                                                                        3 & 3 & 6 & 6 & 8 & 9  \\
                                                                        4 & 7 & 7 & 10 & 10 & 10 \end{array}\right)\;,
\]
with $\alpha_{J_1}(\gotM_1,\gotM_2;\;\gotM_{\lambda,\mu}^{x_5}) = 3$ and $\alpha_{J_2}(\gotM_1,\gotM_2;\;\gotM_{\lambda,\mu}^{x_5}) =1$.
\end{example}

\begin{corollary}\label{cor-zero}
Let $\sigma=L(\gotN)$ be a representation in $\mathcal{C}(\pi_1,\pi_2)\cap \mathcal{D}(\pi_1,\pi_2)$ which has maximal degree $d_\otimes(\pi_1,\pi_2;\;\sigma)$.

Then $\alpha_{J_2}(\gotM_1,\gotM_2;\;\gotN)=0$ holds.

\end{corollary}

\begin{proof}

Let $\sigma_{\max} \in \mathcal{B}(\pi_1,\pi_2)$ be a representation with maximal degree $d_{\max} = d(\pi_1,\pi_2;\;\sigma_{\max})$. Since
$D^{\mathbb{Q}}_{\epsilon(\sigma_{\max})}(\ch(\sigma_{\max})) =1$,
by Proposition \ref{prop-mult} we know that
\[
d_\otimes(\pi_1,\pi_2;\;\sigma_{\max}) =d(\pi_1,\pi_2;\;\sigma_{\max}) = d_{\max}\;.
\]
Hence, since $\sigma_{\max}$ is an element of $\mathcal{C}(\pi_1,\pi_2)\cap \mathcal{D}(\pi_1,\pi_2)$, the maximality condition on $\sigma$ implies $d_{\max} \leq d_\otimes(\pi_1,\pi_2;\;\sigma)$.

Now, let $\sigma'$ be any representation in $\mathcal{C}(\pi_1,\pi_2)$. As exhibited in Section \ref{sec-multi}, there exists a representation $\sigma''\in \mathcal{B}(\pi_1,\pi_2)$ for which $d_\otimes(\pi_1,\pi_2;\;\sigma') = d(\pi_1,\pi_2;\;\sigma'')\leq d_{\max}$.

Thus, $d_\otimes(\pi_1,\pi_2;\;\sigma') \leq d_\otimes(\pi_1,\pi_2;\;\sigma)$ for all $\sigma' \in  \mathcal{C}(\pi_1,\pi_2)$. The statement then follows from Lemma \ref{lem-import}.

\end{proof}

In Example \ref{examp-2}, the maximal representation described in the above corollary is obtained as $L(\gotM_{\lambda,\mu}^{x_1})$.

\section{The case of $\alpha_{J_2}=0$}\label{sec-case}

\subsection{Combinatorial gadgets}\label{subs-comb}

Let us fix a $321$-avoiding permutation $w\in S_n$. Let us also fix a disjoint partition $J_1\cup J_2 = \{1,\ldots,n\}$, such that $w|_{J_1}, w|_{J_2}$ are both increasing functions.

We would like to define algorithmically certain subsets of $S_n$ which depend on the data $(w,J_1,J_2)$. Later in this section we will see the relevance of these procedures to our questions in representation theory.

Let $K  \subset J_2$ be a given subset. We recursively define
\[
w=\sigma^0_K, \sigma^1_K, \ldots , \sigma^n_K = \sigma_K \in S_n\]
in the following manner:

For $1\leq j\leq n$, if $j\not\in K$, we set $\sigma^j_K := \sigma^{j-1}_K$.

Otherwise, let $i\in J_1$ be the minimal index, if exists, for which $j< i$ and $\sigma^{j-1}_K(j) < \sigma^{j-1}_K(i)$. Then, $\sigma^j_K := \sigma^{j-1}_K\cdot (i,j)$\footnote{This will be our notation for a transposition.}. If such $i$ does not exist, we set $\sigma^j_K := \sigma^{j-1}_K$.

Let $e(K)$ denote the number of indices $1\leq j\leq n$ for which $\sigma^j_K \neq \sigma^{j-1}_K$.

We write $\mathcal{A}(w,J_1,J_2)$ for the collection of subsets of $J_2$, for which $e(K) = |K|$ holds.

We also write
\[
\mathcal{S}(w,J_1,J_2) = \{\sigma_K\in S_n\;:\; K\subset J_2\}\;.
\]
For any given set $K\subset J_2$ we can take its subset $K'\subset K$ of indices $j$ which satisfy $\sigma^j_K \neq \sigma^{j-1}_K$. It is then clear that $K'\in\mathcal{A}(w,J_1,J_2)$, $\sigma_{K'} = \sigma_K$ and $e(K) = e(K')$.

Thus, any permutation $\sigma\in \mathcal{S}(w,J_1,J_2)$ can be written in the form $\sigma = \sigma_K$, for $K\in \mathcal{A}(w,J_1,J_2)$.

The following simple observation will be useful to us in forthcoming inductive arguments. For a permutation $s\in S_n$, we let $s^\vee\in S_{n-1}$ be the permutation defined as
\[
s^\vee(i) = \left\{ \begin{array}{ll} s(i+1) & s(i+1) <s(1) \\ s(i+1)-1 & s(i+1) > s(1) \end{array} \right.\;.
\]
For a subset $K\subset \{1,\ldots,n\}$, we also write $K^\vee = (K\setminus \{1\})-1\subset \{1,\ldots,n-1\}$.
\begin{lemma}\label{lem-obs}
For all $K\subset J_2$,
\[
(\sigma_K)^{\vee} = \sigma_{K^\vee}
\]
holds, where $K^\vee$ is considered as an element of $\mathcal{A}((\sigma^1_K)^{\vee}, J_1^\vee, J_2^\vee)$.

\end{lemma}

\begin{proof}
The proof is straightforward once we verify that $((\sigma^1_K)^{\vee}, J_1^\vee, J_2^\vee)$ constitutes combinatorial data. Indeed, $(\sigma^1_K)^{\vee}|_{J_2^\vee} = w^\vee|_{J_2^\vee}$ is increasing.

Next, if $\sigma^1_K = w$ we are done. Otherwise, there is an index $i\in J_1$ for which $\sigma^1_K = w(1,i)$. The minimality condition by which $i$ is chosen makes sure that the function $(\sigma^1_K)^{\vee}|_{J_1^\vee}$, which may differ from $w^\vee|_{J_1^\vee}$ only on $i-1$, remains increasing.

\end{proof}

We denote by $\mathcal{L}(w,J_1,J_2)\subset \mathcal{S}(w,J_1,J_2)$ the subset of $321$-avoiding permutations. We would like to have a clear combinatorial description of $\mathcal{L}(w,J_1,J_2)$. Eventually we end up with the simple algorithm described in Section \ref{sect-intro-alg}. The reader is encouraged to consult the main example appearing in that introductory section while going through the arguments of the current section.

Let us introduce an order $\prec$ on the set of indices $\{1,\ldots,n\}$ which will depend on the data $(w,J_1,J_2)$. Given $1\leq i,j\leq n$, we write

\[
i\prec j \quad\Leftrightarrow \left\{\begin{array}{cc} i<j\mbox{ or }w(i)<w(j) & i\in J_1,\,j\in J_2 \\ i<j\mbox{ and }w(i)<w(j)  &\mbox{otherwise} \end{array}\right.\;.
\]
\begin{proposition}
The relation $\prec$ on $\{1,\ldots,n\}$ is a linear order.
\end{proposition}

The following lemma will give a certain preliminary connection between $\prec$ and $\mathcal{S}(w,J_1,J_2)$.

\begin{lemma}\label{lem-prelim}
Given $K= \{j_1<\ldots <j_b\}\in \mathcal{A}(w,J_1,J_2)$, there are indices $i_1<\ldots < i_b$ in $J_1$, such that for all $1\leq h\leq b$, $j_h\prec i_h$ and $\sigma_K^{j_h} = \sigma^{j_{h}-1}_K \cdot (i_h,j_h)$.

\end{lemma}
\begin{proof}
We prove it by induction on $h$, i.e. we suppose that there are indices $i_1\prec \ldots \prec i_{h-1}$ in $J_1$ as in the statement. Note, that $\sigma_K^{j_{h}-1}(j_h) = w(j_h)$. By construction, $\sigma^{j_h}_K = \sigma^{j_h-1}_K (i_h, j_h)$ for the minimal $i_h\in J_1$ that satisfies $j_h<i_h$ and $w(j_h) < \sigma^{j_h-1}_K(i_h)$.  We need to show that $i_{h-1}<i_h$ (when $h>1$) and that $j_h\prec i_h$.

The latter inequality is clear from the fact the algorithm for $\sigma_K$ gives $\sigma^{j_h-1}_K(i_h) \leq w(i_h)$.

Suppose that $h>1$. Since $\sigma^{j_h-1}_K(i_{h-1}) =\sigma^{j_{h-1}}_K(i_{h-1}) = \sigma^{j_{h-1}-1}_K(j_{h-1}) = w(j_{h-1}) < w(j_h)$, we see that $i_h\neq i_{h-1}$. Hence, $\sigma_K^{j_{h-1}-1}(i_{h}) = \sigma_K^{j_{h-1}}(i_{h})$. Thus, from
\[
\left\{\begin{array}{ll} j_{h-1}< j_h< i_h \\ w(j_{h-1}) < w(j_h) < \sigma^{j_h-1}_K(i_{h})=\sigma^{j_{h-1}}_K(i_{h}) =\sigma^{j_{h-1}-1}_K(i_h) \end{array} \right.
\]
 and the minimality condition which defines $i_{h-1}$, we see that $i_{h-1}\leq i_h$.

\end{proof}

We define a function $f:J_2\to J_1\cup \{0\}$ in the following manner. For $j\in J_2$, $f(j)$ will equal the minimal element in $J_1$, if exists, for which $j\prec f(j)$ and $|[j,f(j)]_\prec\cap J_1|=|[j,f(j)]_\prec\cap J_2|$ hold. Here $[,]_\prec$ denotes a closed interval under the linear $\prec$ relation. We set $f(j)=0$ if such element does not exist.

We write $\widetilde{J}\subset J_2$ for the set of indices $j\in J_2$ for which $f(j)\neq 0$. For every $j\in\widetilde{J}$, we then can define the sets $F(j) = [j,f(j)]_\prec$, $G(j) = [j,f(j)]_\prec\cap J_2$ and $H(j) = [j,f(j)]_\prec\cap J_1$.

Let us write $G(j) = \{j = j_1 \prec j_2 \prec \ldots \prec j_l\}$ and $H(j) = \{i_1 \prec i_2 \prec \ldots \prec i_l\}$. We define the permutation $\gamma(j) = (i_1,j_1)(i_2,j_2) \cdots  (i_l,j_l)\in S_n$.

Given a subset $L\subset \widetilde{J}$, for which the intervals $\{F(j)\}_{j\in L}$ are pairwise disjoint, we then define the permutation $w^L = w\prod_{j\in L} \gamma(j)\in S_n$. Since the intervals are disjoint, the product in the above definition is commutative.

See the detailed example in Section \ref{sect-intro-alg} for an instructive description of the algorithm that produces $w^L$ out of a given combinatorial data and a choice of $L$.

\begin{lemma}\label{lem-disjoint}
For every two induces $j_1 < j_2$ in $\widetilde{J}$, either $F(j_1)\cap F(j_2)=\emptyset$ or $F(j_2)\subset F(j_1)$.
\end{lemma}
\begin{proof}
Assume that $F(j_1)\cap F(j_2)\neq\emptyset$. Since both sets are intervals for $\prec$, we know that $j_2\in F(j_1)$. Hence, $|[j_1,j_2)_\prec\cap J_2|>|[j_1,j_2)_\prec\cap J_1|$. Yet, for every $k\in F(j_2)$ we know that  $|[j_2,k]_\prec\cap J_2|\geq|[j_2,k]_\prec\cap J_1|$ by definition. Thus, $|[j_1,k]_\prec\cap J_2|>|[j_1,k]_\prec\cap J_1|$ which means that $f(j_2)\prec f(j_1)$.
\end{proof}

A particular consequence of the above lemma is that for any subset $L\subset \widetilde{J}$, there exists a unique subset $\widetilde{L}\subset L$ such that $\cup_{j\in L} F(j) = \cup_{j\in \widetilde{L}} F(j)$ and the latter union is disjoint. It will be convenient to define $w^L=w^{\widetilde{L}}$ for any $L\subset \widetilde{J}$.

\begin{lemma}\label{lem-concat}
Suppose that $\sigma_1,\sigma_2\in S_n$ are two $321$-avoiding permutations, for which there is an integer $1\leq k\leq n$ so that the set $\{1,\ldots, k\}$ is invariant under $\sigma_1^{-1}\sigma_2$. Then, the permutation given as
\[
\sigma(i) = \left\{ \begin{array}{cc} \sigma_1(i) & 1\leq i\leq k \\ \sigma_2(i) & k<i\leq n \end{array}\right.
\]
is $321$-avoiding.
\end{lemma}

\begin{proof}
Suppose that $\sigma$ contains a $321$ pattern. Hence, there are $1\leq i_1<i_2<i_3\leq n$ for which $\sigma(i_3)<\sigma(i_2)<\sigma(i_1)$. If $k<i_1$ or $i_3\leq k$ were to hold, this would contradict the pattern avoidance condition on $\sigma_2$, or $\sigma_1$, respectively. Hence, we have $i_1\leq k$ and $k<i_3$.

In case $k<i_2$ holds, we have $\sigma_2(i_2)>\sigma_2(i_3)$, but $\sigma_1(i_1) = \sigma_2(j)$ for some $1\leq j\leq k$ from the assumption. Thus, the triple $j<i_2<i_3$ contradict the pattern avoidance condition for $\sigma_2$. Otherwise, we have $i_2\leq k$, $\sigma_1(i_1)> \sigma_1(i_2)$ and $\sigma_2(i_3) = \sigma_1(j')$ for some $k<j'\leq n$ by the assumption. Similarly, $i_1<i_2<j'$ gives a contradiction the pattern avoidance condition of $\sigma_1$.

\end{proof}

\begin{proposition}\label{prop-avoid}
For any $L\subset \widetilde{J}$, the permutation $w^L$ is $321$-avoiding.
\end{proposition}

\begin{proof}
Let us first give a proof for the case that $|\widetilde{L}|=1$, that is, $\cup_{j\in L} F(j) = F(j_1)$ for some $j_1\in \widetilde{J}$.

We write $J = \{j'_1 < \ldots < j'_{c'} < j_1 < \ldots < j_c < j''_1 < \ldots < j''_{c''}\}$ and $I = \{i'_1 < \ldots < i'_{d'} < i_1 < \ldots < i_c < i''_1 < \ldots < i''_{d''}\}$, for some $c',c'',d',d''\geq 0$, $c = |G(j_1)| = |H(j_1)|$ and $i_c = f(j_1)$, i.e. $G(j_1) = \{j_1,\ldots,j_c\}$ and $H(j_1) = \{i_1,\ldots,i_c\}$. Note, that $j_1\prec i_1$ and $j_c\prec i_c$ imply the inequalities $j_1<i_1,\, w(j_1)<w(i_1),\,j_c < i_c,\, w(j_c) < w(i_c)$.

It is enough to show that $w^L$ can be decomposed into two ascending sequences. Recall that $w^L = w(i_1,j_1)\cdots(i_c,j_c)$.

Suppose that $w(i'_{d'})< w(j_1)$ holds. Since $w^L$ and $w$ are equal on $\{j'_h, i'_f\}_{1\leq h \leq c',\, 1\leq f\leq d'}$ and $w(j'_1) < \ldots <w(j'_{c'}) < w(j_1),\; w(i'_1) < \ldots < w(i'_{d'})$ hold by the combinatorial data condition, we see that in that case
\[
w^L(j'_1) <  \ldots < w^L(j'_{c'}) < w(i_1) =w^L(j_1) < \ldots < w(i_c)=w^L(j_c)\;,
\]
\[
w^L(i'_1) <  \ldots < w^L(i'_{d'}) < w(j_1) =w^L(i_1) < \ldots < w(j_c)=w^L(i_c),
\]
which give $j'_1 < \ldots < j'_{c'} < j_1 < \ldots < j_c$ and $i'_1 < \ldots < i'_{d'} < i_1 < \ldots < i_c$ as two ascending sequences in $w^L$.

Otherwise, $w(j_1) < w(i'_{d'})$. Since $i'_{d'}\prec j_1$, we must have $i'_{d'}<j_1$ in that case. Therefore, the sequences $j'_1 < \ldots < j'_{c'} < i_1 < \ldots < i_c$ and $i'_1 < \ldots < i'_{d'} < j_1 < \ldots < j_c$ form two ascending sequences in $w^L$.

With a symmetric argument, it is easy to see that when $j''_1 < i_c$, the sequences $j_c< j''_1 < \ldots < j''_{c''}$ and $i_c< i''_1 < \ldots < i''_{c''}$ are ascending in $w^L$. While otherwise, the sequences $i_c< j''_1 < \ldots < j''_{c''}$ and $j_c< i''_1 < \ldots < i''_{c''}$ fulfil this condition.

Altogether, we have covered $w^L$ by two ascending sequences in all cases.

Now, for a general $L\subset \widetilde{J}$, by Lemma \ref{lem-disjoint} we can write $\cup_{j\in L} F(j) = \cup_{h=1}^k F(l_h)$, where $l_1<\ldots < l_k$ are indices in $L$ and the latter union is disjoint. Let us write $L_i = \{l_1,\ldots,l_i\}$, for all $1\leq i \leq k$, and $L_0 = \emptyset$,. Then, $w^{L_k} = w^L$ and $w^{L_0} = w$.

We prove by induction that $w^{L_i}$ is $321$-avoiding, for all $1\leq i\leq k$. Denote $z_i= w^{\{l_i\}}$. We have proved that $z_i$ is $321$-avoiding. Note, that since $F(l_i)$ and $\cup_{h=1}^{i-1} F(l_h)$ are disjoint, the permutation $z_i^{-1}w^{L_{i-1}}$ leaves the set $\{1,\ldots, f(l_{i-1})\}$ invariant. Hence, $w_i$ is $321$-avoiding as well, by Lemma \ref{lem-concat} and the induction hypothesis.

\end{proof}

\begin{proposition}\label{prop-charl}
We have the equality
\[
\mathcal{L}(w,J_1,J_2) = \{ w^L\;:\; L\subset \widetilde{J} \}\;.
\]
For a subset $L\subset \widetilde{J}$ as above, $w^L = \sigma_K$, where $K = \cup_{j\in L} G(j)\in \mathcal{A}(w,J_1,J_2)$.

\end{proposition}

\begin{proof}
Let $L= \{j_1\prec \ldots\prec j_l\}\subset \widetilde{J}$ be a subset. By the remark after Lemma \ref{lem-disjoint}, we are free to assume that $\{F(j_s)\}_{s=1}^l$ are pairwise disjoint. By Proposition \ref{prop-avoid}, $w^L$ is $321$-avoiding. Hence, in order to prove that $w^L\in \mathcal{L}(w,J_1,J_2)$, it is enough to show that $w^L\in \mathcal{S}(w,J_1,J_2)$. More precisely, we will show that $w^L = \sigma_K$, where $K$ is as in the statement.

Let us write further $G(j_s) = \{j_s = j_s^1\prec\ldots \prec j_s^{|G(j_s)|}\}$ and $H(j_s) = \{i_s^1\prec\ldots \prec i_s^{|G(j_s)|}\}$. The order on the set $K$ is then naturally parametrized by the pairs $\{(s,t)\}_{1\leq s \leq l,\, 1\leq t\leq |G(j_s)|}$ with their lexicographical order. We write $\sigma^{(s,t)}_K\in S_n$ for the corresponding iteration in the construction of $\sigma_K$.

We will show inductively that $\sigma^{(s,t)}_K = \sigma^{(s,t-1)}_K (i_s^t, j_s^t)$, for all $(s,t)$, where we take $(s,0) = (s-1,|G(j_{s-1}|)$ for $s>1$ and $\sigma^{(1,0)}_K=\sigma^0_K = w$, for ease of notation.

By construction, $\sigma^{(s,t)}_K = \sigma^{(s,t-1)}_K (i', j_s^t)$ for the minimal $i'\in J_1$ that satisfies $j_s^t<i'$ and $w(j_s^t) < \sigma^{(s,t-1)}_K(i')$. From the induction hypothesis and Lemma \ref{lem-prelim}, we know that $j_s^t\prec i'$ and that, when $t>1$, $i_s^{t-1}< i'$. From the construction of $H(j_s)$ it is clear these conditions imply $i_s^t\leq i'$.

On other other hand, it follows from the definition of $f(j_s)$ that $j_s^t \prec i_s^t$. Thus, $j_s^t < i_s^t$ and $w(j_s^t) < \sigma^{(s,t-1)}_K ( i_s^t)$ hold. It follows that $i'=i_s^t$.

Conversely, suppose that $K\in \mathcal{A}(w,J_1,J_2)$ is such that $\sigma_K\in \mathcal{L}(w,J_1,J_2)$. We need to show that $K = \cup_{j\in L}G(j)$ for a subset $L\subset \widetilde{J}$. Since $j\in G(j)$ for all $j\in \widetilde{J}$, it enough to show that $K\subset \widetilde{J}$ and that $G(j)\subset K$ for all $j\in K$.

Let $j\in K$ be an index. Let $j<j'\in J_2$ be the minimal index such that $j'\not\in K$, if exists. Otherwise, let us treat $j'$ as $+\infty$ for the $<$ relation on $J_2$ and the $\prec$ relation on $\{1,\ldots,n\}$.
Let us write $J_2\cap [j,j') = \{j= j_1 < \ldots < j_k\}$. By Lemma \ref{lem-prelim}, there are indices $i_1<\ldots<i_k$ in $J_1$, such that $j_h\prec i_h$, for all $1\leq h\leq k$. In particular, $j\prec i_h$.

We claim that $i_k\prec j'$. Assume the contrary. Then, $j_k< j'<i_k$ and $w(j_k)< w(j') < w(i_k)$ hold. Yet, by Lemma \ref{lem-prelim}, $\sigma_K(i_k) = w(j_k)$ and $\sigma_K(j_k)=w(i_k)$. This gives a contradiction to the $321$ pattern avoiding property of $\sigma_K$.

Hence, $k\leq |J_1\cap [j,j')_\prec|$. This implies that $j\in \widetilde{J}$ and that $f(j)\prec j'$. In particular, $G(j)\subset K$.

\end{proof}

\subsection{Representation theory}

Let us fix again a regular pair of ladder representations $\pi_1 = L(\gotM_1),\pi_2=L(\gotM_2)\in \irr^r$, together with its combinatorial data $(w,J_1,J_2)$, so that $w\in S_n$. 
We write $\gotM_1+ \gotM_2 = \gotM^w_{\lambda,\mu}$, for $\lambda,\mu\in \mathcal{P}_n$.

For a permutation $x\in S_n$, we will write $\Pi(x):= L(\gotM_{\lambda,\mu}^x)\in \irr^r$ in this subsection.

Let us write $\mathcal{C}^0(\pi_1,\pi_2)\subset \mathcal{C}(\pi_1,\pi_2)$ for the subset of representations $\sigma= L(\gotN)$, for which $\alpha_{J_2}(\gotM_1,\gotM_2,\gotN)=0$ holds.

\begin{proposition}\label{prop-techn}
The equality
\[
\Pi(\mathcal{S}(w,J_1,J_2)) = \mathcal{C}^0(\pi_1,\pi_2)
\]
holds. Moreover, for every subset $K\subset J_2$ such that $K\in \mathcal{A}(w,J_1,J_2)$, we have
\[
d_\otimes(\pi_1,\pi_2;\;\Pi(\sigma_K)) = \alpha_{J_1}(\gotM_1,\gotM_2;\;\gotM_{\lambda,\mu}^{\sigma_K}) = |K|\;.
\]

\end{proposition}

\begin{proof}
We prove by induction on $n$.

Suppose that $\sigma \in \mathcal{C}^0(\pi_1,\pi_2)$ is given. By Proposition \ref{prop-matrix} we know that $\sigma = \Pi(x)$, for a permutation $x\in S_n$.

Let $C = C(\pi_1,\pi_2;\;\sigma)=(c_i^j)_{i=1,\ldots, n}^{j=1,\ldots, n+1}$ be the matrix supplied by Proposition \ref{prop-matrix}.

Let us consider the truncated matrix $C' = (c_i^j)_{i=2,\ldots, n}^{j=1,\ldots, n}$. Again (as in the proof of Proposition \ref{prop-comb}), we can verify that $C' = C(\gotM'_1, \gotM'_2;\; \gotN')$, where $\gotN' = \sum_{j=2}^{n} [\lambda_j, \mu_{x(j)}]$, $\gotM'_1 = \sum_{i\in J_1\setminus \{1\} } [\lambda_i, c_i^n -1]$ and $\gotM'_2 = \sum_{i\in J_2\setminus \{1\} } [\lambda_i, c_i^n -1]$. In particular, $\sigma' = L(\gotN')\in \mathcal{C}(L(\gotM'_1), L(\gotM'_2))$.

We write $\lambda^\vee = (\lambda_2,\ldots,\lambda_n), \mu^\vee = (\mu_1,\ldots,\check{\mu}_{x(1)},\ldots,\mu_n)\in \mathcal{P}_{n-1}$. Then, $\gotN' = \gotM^{x^\vee}_{\lambda^\vee,\mu^\vee}$, and $\gotM'_1 + \gotM'_2 = \gotM^y_{\lambda^\vee,\mu^\vee}$, for a certain $y\in S_{n-1}$. Hence, the combinatorial data of the pair $L(\gotM'_1), L(\gotM'_2)$ is given by $(y, J^\vee_1,J^\vee_2)$. Moreover, there exists a permutation $\overline{y}\in S_{n}$, such that $(\overline{y})^\vee = y$ and $c^n_i = \mu_{\overline{y}(i)}+1$, for all $2\leq i\leq n$.

By comparison of $C$ and $C'$, we clearly see that $0\leq \alpha_{J_2}(\gotM'_1,\gotM'_2;\;\gotN')\leq \alpha_{J_2}(\gotM_1,\gotM_2;\;\gotM^x_{\lambda,\mu})$, which implies $\sigma'\in \mathcal{C}^0(L(\gotM'_1), L(\gotM'_2))$. Thus, by the induction hypothesis, we have $x^\vee = \sigma_{K'}$, for a set $K'\in \mathcal{A}(y, J^\vee_1,J^\vee_2)$.

The condition of $\alpha_{J_2}(\gotM_1,\gotM_2;\; \gotM^x_{\lambda,\mu})= 0$ forces that for all $i\in J_2$ and all $1\leq j\leq n+1$, $c_i^j \in\{ \lambda_i, \mu_{w(i)}+1\}$. In particular, combined with the other conditions imposed on the entries of $C$, we see that either $[\lambda_1, \mu_{x(1)}] = [c^n_{i_1}, c^{n+1}_{i_1} -1]$ for $i_1\in J_1$, or $[\lambda_1, \mu_{x(1)}] =[c^n_{i_2},  c^{n+1}_{i_1} -1]$ for $i_1\in J_1$ and $i_2\in J_2$, such that $c^n_{i_2}<c^{n+1}_{i_2} = c^n_{i_1}$.

In the former case, by Remark \ref{rmrk-matrix}, $c^n_{i_1} = \lambda_{i_1}$. It then follows that $i_1=1$ and that $x(1) = w(i_1) = w(1)$ (case I).

In the latter case, we similarly have $i_2=1$. Then, either $x(1) = w(i_2)= w(1)$ (case I) or $x(1) = w(i_1)$ (case II).

Note, that in case II, since $c^{n+1}_{i_2} < c^{n+1}_{i_1}$, we must have $w(1) < w(i_1)$. Furthermore, for any $i_1>i\in J_1$, we know that $c^{n+1}_i< c^n_{i_1} = c^{n+1}_{i_2}$. Hence, $i_1$ is the minimal index in $J_1$, for which $w(1)<w(i_1)$ holds.

Altogether, we see that $x(1) = \sigma_K(1)$, where $K\in \mathcal{A}(w,J_1,J_2)$ can be taken as $K = K'+1$ in case I, or $K = \{1\}\cup(K'+1)$ in case II.

Furthermore, since in all cases $c^n_i = c^{n+1}_i$, for any $i\not\in \{i_1,i_2\}$, we can write $c^n_i = \mu_{\sigma^1_K(i)} +1$, for all $2\leq i\leq n$. Hence, we can choose $\overline{y} = \sigma^1_K$. In particular, we obtain that $(\sigma^1_K)^\vee = y$. Since $K^\vee = K'$, Lemma \ref{lem-obs} now implies $x^\vee = (\sigma_K)^\vee$.

Thus, $x=\sigma_K$ and $\sigma\in \Pi(\mathcal{S}(w,J_1,J_2))$. Moreover, by comparing $C$ and $C'$ we see that

\[
 \alpha_{J_1}(\gotM_1,\gotM_2;\;\gotM_{\lambda,\mu}^{\sigma_K}) = \left\{\begin{array}{ll}   \alpha_{J_1}(\gotM'_1,\gotM'_2;\;\gotN')  & \mbox{case I} \\
  \alpha_{J_1}(\gotM'_1,\gotM'_2;\;\gotN') +1 & \mbox{case II}\end{array}\right.\;.
\]
The identity in the statement then follows by induction, since $|K|= |K'|$ in case I, and $|K|= |K'|+1$ in case II.
\\ \\
We are left to show the converse statement. Suppose now that $K\in \mathcal{A}(w,J_1,J_2)$ is given. We need to show that $\Pi(\sigma_K)\in \mathcal{C}^0(\pi_1,\pi_2)$.

By Lemma \ref{lem-obs}, we know that $(\sigma_K)^\vee\in \mathcal{A}((\sigma^1_K)^\vee, J^\vee_1, J^\vee_2)$. Hence, by the induction hypothesis, we have $L(\gotN')\in \mathcal{C}^0(L(\gotM'_1), L(\gotM'_2))$, where
\[
\gotN' = \sum_{j=2}^{n} [\lambda_j, \mu_{\sigma_K(j)}],\quad\gotM'_1 = \sum_{i\in J_1\setminus \{1\} } [\lambda_i, \mu_{\sigma^1_K(i)}],\quad \gotM'_2 = \sum_{i\in J_2\setminus \{1\} } [\lambda_i, \mu_{\sigma^1_K(i)}]\;.
\]
Let $C' =  C(\gotM'_1, \gotM'_2;\; \gotN')=(c_i^j)_{i=2,\ldots, n}^{j=1,\ldots, n}$ be the corresponding matrix. We claim that the matrix $C:= (c_i^j)_{i=1,\ldots, n}^{j=1,\ldots, n+1}$ satisfies $C = C(\gotM_1, \gotM_2;\; \gotM^{\sigma_K}_{\lambda,\mu})$, where the remaining entries are defined as $c_1^1 = \ldots = c_1^n = \lambda_1$ and $c^{n+1}_i = \mu_{w(i)}+1$.

The required properties in Proposition \ref{prop-matrix} are easily verified from the facts that $c^n_i = \mu_{\sigma^1_K(i)}$ for all $2\leq i\leq n$ and that either $\sigma^1_K = w$ or $\sigma^1_K = w\cdot(1,i_1)$, for an index $i_1\in J_1$, holds. It is also evident from construction that $\alpha_{J_2}(\gotM_1,\gotM_2;\;\gotM_{\lambda,\mu}^{\sigma_K}) = \alpha_{J_2}(\gotM'_1,\gotM'_2;\;\gotN') = 0$.

\end{proof}

\begin{corollary}\label{cor-fin}
  The equality
\[
\Pi(\mathcal{L}(w,J_1,J_2)) = \mathcal{C}^0(\pi_1,\pi_2)\,\cap\,\mathcal{D}(\pi_1,\pi_2)
\]
holds.
\end{corollary}

We would like to draw attention on the following consequence of Corollary \ref{cor-fin} and Theorem \ref{thm-main}\footnote{Theorem \ref{thm-main} will, of course, be proved in the next section without relying on Theorem \ref{thm-side}.}. It provides a direct method for constructing certain subquotients for a given product of a pair of regular ladder representations.
\begin{theorem}\label{thm-side}
\[
\Pi(\mathcal{L}(w,J_1,J_2)) \subset \mathcal{B}(\pi_1,\pi_2)\;.
\]
\end{theorem}

We can now finish with the key result of this work.

\begin{theorem}\label{thm-key}
Let $\pi_1,\pi_2\in \irr^r$ be a regular pair of ladder representations as above. Then,
\begin{enumerate}

\item There is a unique representation $\pi_{\max}\in \mathcal{C}(\pi_1,\pi_2)\,\cap\,\mathcal{D}(\pi_1,\pi_2)$ with a maximal degree $d_\otimes(\pi_1,\pi_2;\;\pi_{\max})$.

When writing $\pi_{\max} = \Pi(x_{\max})$, for a permutation $x_{\max}\in S_n$, we have $x_{\max} = w^{\widetilde{J}}$, where $\widetilde{J}\subset J_2$ is the set described in Subsection \ref{subs-comb}.

In particular, $x_{\max}$ depends only on the data $(w,J_1,J_2)$.

\item We have $\pi_{\max}\in \mathcal{B}(\pi_1,\pi_2)$.

In fact, $d(\pi_1,\pi_2;\;\pi_{\max})$ is maximal among the degrees in $\mathcal{B}(\pi_1,\pi_2)$ and $\pi_{\max}$ is the unique representation for which this maximum is attained.

\item For any $321$-avoiding $x\in S_n$, $m(\Pi(x)_\otimes, \mathbf{r}_{\Pi(x)}(\pi_{\max}))=1$ holds, if and only if, $x=x_{\max}$.

\end{enumerate}

\end{theorem}
\begin{proof}
  \begin{enumerate}
    \item

    By Corollary \ref{cor-zero}, it is enough to prove the statement for $\mathcal{C}^0(\pi_1,\pi_2)\,\cap\,\mathcal{D}(\pi_1,\pi_2)$. Now, by Corollary \ref{cor-fin} and Proposition \ref{prop-charl}, the set of representations in $\mathcal{C}^0(\pi_1,\pi_2)\,\cap\,\mathcal{D}(\pi_1,\pi_2)$ is given by $\{\Pi(w^L)\}_{L\subset \widetilde{J}}$.

        Moreover, Proposition \ref{prop-charl} also states that $w^L$ depends only on the set $K(L): =  \bigcup_{j\in L} G(j)\in\mathcal{A}(w,J_1,J_2)$, while Proposition \ref{prop-techn} says that $d_\otimes(\pi_1,\pi_2;\;\Pi(w^L)) =|K(L)|$. Hence, we need to show that there is a unique largest set among the sets $\{K(L)\}_{L\subset \widetilde{J}}$. Clearly $K(\widetilde{J})$ is such a set.


    \item Let $\pi\in\mathcal{B}(\pi_1,\pi_2)$ be a representation with maximal degree $d_{\max}=d(\pi_1,\pi_2;\;\pi)$. As shown in the proof of Corollary \ref{cor-zero}, we have $d_\otimes(\pi_1,\pi_2;\;\pi)= d_{\max}$ and $d_\otimes(\pi_1,\pi_2;\;\sigma)\leq d_{\max}$, for all $\sigma\in \mathcal{C}(\pi_1,\pi_2)$.

        It now follows from (1) that $\pi\cong \pi_{\max}$ and that the maximality condition in the statement holds.

    \item From Proposition \ref{prop-mult} and Lemma \ref{lemma-first}, we see that if $m(\Pi(x)_\otimes, \mathbf{r}_{\Pi(x)}(\pi_{\max}))=1$ holds, then $d_\otimes( \pi_1,\pi_2;\; \Pi(x)) = d_{\max}$. The statement then follows from the uniqueness of maximal elements in (1) and (2).

  \end{enumerate}

\end{proof}

\begin{proposition}\label{conj-last}
  For every regular pair of ladder representations $\pi_1,\pi_2$, the representation $\pi_{\max}\in \mathcal{B}(\pi_1,\pi_2)$ of Theorem \ref{thm-key} is the unique irreducible sub-representation of $\pi_1\times \pi_2$.
\end{proposition}
\begin{proof}
The decomposition problem for products of dual canonical basis elements can be categorified into the representation theory of KLR algebras of type $A$. In fact, using the isomorphisms (up to certain completions) described in \cite{brun-kles, rouq}, we obtain an equivalence of categories of representations of affine Hecke algebras of type $A$ and suitable (graded) KLR algebras.

In particular, it is enough to prove this proposition in the KLR setting. This version follow from Theorem \ref{thm-key}(2) and \cite[Lemma 7.5]{McNm} (see also \cite[Section 4.2]{kkkoI}).

\end{proof}

\section{Robinson-Schensted correspondence}\label{sec-rs}

Given combinatorial data $(w,J_1,J_2)$ with $w\in S_n$, we will write $(w,J_1,J_2)_{\max}:=x_{\max}\in S_n$ for the permutation supplied by Theorem \ref{thm-key}(1).

Note, that Theorem \ref{thm-key}(3) proves the statement of Theorem \ref{thm-main2} for the cases in which $w'\in S_n$ can be presented as $w' = (\overline{w},J_1,J_2)_{\max}$, for some combinatorial data $(\overline{w},J_1,J_2)$. In order to prove Theorem \ref{thm-main2}, we are left with the question of whether such data can be found for every $321$-avoiding permutation $w'$.

We will answer this question positively by making use of the Robinson-Schensted correspondence.

Recall that the RS algorithm attaches to each permutation in $S_n$ a pair of standard Young tableaux of same shape on $n$ squares. This process is known to be bijective.

It will be convenient for our needs to revise somewhat the common conventions for the algorithm. More precisely, if we let $w\mapsto (\underline{P}(w),\underline{Q}(w))$ denote the RS algorithm with the common conventions (for example, the ones used in \cite{fulton-tabl}), we will write $(P(w),Q(w))$ for the pair $(\underline{P}(w_0w w_0),\underline{Q}(w_0w w_0))$ taken with the numberings in their squares permuted again by $w_0$. Here $w_0= (n\,n-1\ldots 1)\in S_n$ denotes the longest permutation.

To be more precise, in what follows we will describe directly the inverse algorithm, that is, the algorithm for producing a permutation $w(P,Q)$ out of a pair of standard Young tableaux $(P,Q)$ of same shape.

It is well known that $321$-avoiding permutations correspond to pairs of Young tableaux whose shape has at most two rows. Since this will be the case of interest for us, let us restrict our description of the algorithm to the $2$-row case.

\begin{definition}

 a \textbf{standard Young tableaux on $n$ squares of a $2$-row shape} is a disjoint partition $I_1\cup I_2 =\{1,\ldots,n\}$ into $2$ rows, such that when writing
\[
I_1 = \{p_{1,1}> p_{1,2} \ldots > p_{1,k_1}\}\;,
\]
\[
I_2 = \{p_{2,1}> p_{2,2} \ldots> p_{2,k_2}\}\;,
\]
we have $k_1>k_2$, and $p_{1,i}> p_{2,i}$, for all $1\leq i\leq k_2$.

We will write a standard Young tableaux as $P = (p_{i,j})$.

\end{definition}

\textbf{(Inverse) RS algorithm:}
Suppose that $(P,Q)$ is a given a pair of standard Young tableaux on $n$ squares of same $2$-row shape.

On the first step we produce a tableau $Q'$ on $n-1$ squares by removing the entry $1$ from $Q$, and subtracting by $1$ all other entries. The removed square was in the end of either the first or the second row.

In case it is the first row, we take note of the value $k$ in the end of the first row of $P$, remove it and subtract all entries larger than $k$ by $1$ to produce the tableau $P'$ on $n-1$ entries.

In the latter case, we take note of the value $l$ in the end of the second row of $P$. We then find the smallest entry $k$ in the first row of $P$ for which $l<k$. We then produce $P'$ by first removing the entry of $l$ in the second row of $P$, replacing $k$ with $l$ in the first row, and subtracting all entries larger than $k$ by $1$.

We now determine $w = w(P,Q)$ inductively by setting $w(1)=k$, and $w^\vee= w(P',Q')\in S_{n-1}$. Here we use the $^\vee:S_n\to S_{n-1}$ operation as defined in Subsection \ref{subs-comb}.
\\ \\
For a $321$-avoiding permutation $w\in S_n$, we define $\overline{w}\in S_n$ to be the permutation constructed by setting $\overline{w}(q_{c,d})=p_{c,d}$ for all indices $c,d$, where $((p_{c,d}),(q_{c,d})) = (P(w),Q(w))$. We also let $J_1(w)\cup J_2(w) = \{1,\ldots,n\}$ be the partition determined by the rows of $Q(w)$.

It is easy to see that $(\overline{w}, J_1(w),J_2(w))$ gives combinatorial data.

The following observation is immediate from a comparison of the above inverse RS algorithm with the algorithm for $\sigma_K$ in Subsection \ref{subs-comb}.

\begin{observation}\label{obs}
Let $K\in \mathcal{A}(\overline{w}, J_1(w),J_2(w))$ be a set which satisfies $\{1\}\cap J_2(w)\subset K$. Then, $\sigma_K(1) = \sigma^1_K(1)= w(1)$ and $(\sigma^1_K)^\vee = \overline{w^\vee}$.
\end{observation}

\begin{lemma}\label{lem-compl}
For any $321$-avoiding permutation $w\in S_n$, we have $\widetilde{J} = J_2(w)$, where $\widetilde{J}$ is defined as in Subsection \ref{subs-comb} for the data $(\overline{w}, J_1(w),J_2(w))$.
\end{lemma}
\begin{proof}
We write $((p_{c,d}),(q_{c,d})) = (P(w),Q(w))$. Let $\prec$ denote the relation on $\{1,\ldots,n\}$ defined using the data of $(\overline{w},J_1(w),J_2(w))$. Let $f:J_2(w)\to J_1(w)\cup\{0\}$ be the function defined in Subsection \ref{subs-comb}.

Let $j\in J_2(w)$ be a given index. Then, $j = q_{2,d_j}$, for an index $1\leq d$. Note, that by the Young tableau condition, for all $1\leq d\leq d_j$, we have
\[
q_{1,1}> q_{2,d} \geq q_{2,d_j},\quad \overline{w}(q_{1,1})>\overline{w}(q_{2,d}) \geq \overline{w}(q_{2,d_j}),
\]
\[
q_{1,1}\geq q_{1,d} > q_{2,d_j},\quad \overline{w}(q_{1,1})\geq\overline{w}(q_{1,d}) > \overline{w}(q_{2,d_j})\;.
\]
Hence, $q_{1,d},q_{2,d}\in [j,q_{1,1}]_{\prec}$, for all $1\leq d\leq d_j$. Moreover,
\[
\left| [j,q_{1,1}]_{\prec}\cap J_2(w)\right| = \left|\{q_{2,1}, \ldots, q_{2,d_j}\}\right| = \left|\{q_{1,1}, \ldots, q_{1,d_j}\}\right| \leq \left| [j,q_{1,1}]_{\prec}\cap J_1(w)\right|\;,
\]
which implies that $f(j)\neq 0$.

\end{proof}

\begin{proposition}\label{prop-RS}
Let $w\in S_n$ be a $321$-avoiding permutation.

Then, $w = (\overline{w}, J_1(w),J_2(w))_{\max}$.

\end{proposition}

\begin{proof}
We give a proof by induction on $n$.

By Lemma \ref{lem-compl}, Proposition \ref{prop-charl} and the proof of Theorem \ref{thm-key}, we know that
\[
(\overline{w}, J_1(w),J_2(w))_{\max} = \sigma_{J_2(w)}\;,
\]
where $J_2(w)$ is treated as an element of $\mathcal{A}(\overline{w}, J_1(w),J_2(w))$.

By Observation \ref{obs}, we know that $w(1) = \sigma_{J_2(w)}(1)$, Thus, we are left to show that $w^\vee = (\sigma_{J_2(w)})^\vee$.

Now, by Lemma \ref{lem-obs}, $(\sigma_{J_2(w)})^\vee = \sigma_{J_2(w)^\vee}$ holds, where $J_2(w)^\vee$ is taken as an element of $((\sigma^1_{J_2(w)})^\vee, J_1(w)^\vee, J_2(w)^\vee)$. Yet, from Observation \ref{obs} and the RS algorithm, we see that latter data is in fact equal to $(\overline{w^\vee}, J_1(w^\vee), J_2(w^\vee))$. Thus, by the induction hypothesis,
\[
\sigma_{J_2(w)^\vee} = \sigma_{J_2(w^\vee)} = (\overline{w^\vee}, J_1(w^\vee), J_2(w^\vee))_{\max} = w^\vee\;.
\]

\end{proof}

\begin{corollary}
  Theorems \ref{thm-main} and \ref{thm-main2} hold.
\end{corollary}
\begin{proof}
  By Proposition \ref{prop-redu}, it is enough to prove Theorem \ref{thm-main2}.

  Let $\lambda,\mu\in \mathcal{P}_n$ be regular tuples with $\lambda_n\leq \mu_1$. Let $w,w'\in S_n$ be two $321$-avoiding permutations. Suppose that $m(\sigma_\otimes, \mathbf{r}_{\alpha_\sigma}(\sigma'))>0$, where $\sigma = L(\gotM^w_{\lambda,\mu})$ and $\sigma' = L(\gotM^{w'}_{\lambda,\mu})$. We need to show that $w=w'$.

Let $(\overline{w},J_1,J_2)$ be the combinatorial data such that $w' = (\overline{w},J_1,J_2)_{\max}$, which exists by Proposition \ref{prop-RS}. Let us define the pair of regular ladder representations:
\[
\pi_1^s = L \left( \sum_{i\in J_1} [s +\lambda_i, s+\mu_{\overline{w}(i)}]\right),\quad \pi_2^s = L \left( \sum_{i\in J_2} [s+\lambda_i, s+\mu_{\overline{w}(i)}]\right)\;,
\]
for a choice of integer $s$. Note, that for large enough choice of $r,s$, we will have $\pi_1^s,\pi_2^s\in \irr^r$.

The data $(w,J_1,J_2)$ then becomes the combinatorial data of the regular pair $\pi_1^s,\pi_2^s$. In particular, $\sigma'\nu^s = \pi_{\max} \in \mathcal{B}(\pi_1^s,\pi_2^s)$ in the notation of Theorem \ref{thm-key}.

Since we clearly have $m(\sigma_\otimes, \mathbf{r}_{\alpha_\sigma}(\sigma')) = m((\sigma\nu^s)_\otimes, \mathbf{r}_{\alpha_\sigma}(\sigma'\nu^s))$, the desired equality follows from Theorem \ref{thm-key}(3).

\end{proof}

Applying Proposition \ref{conj-last}, we obtain the following.
\begin{corollary}
  For every $\pi = L(\gotM^x_{\lambda,\mu})$, where $\lambda,\mu\in \mathcal{P}_n$ regular tuples with $\lambda_n\leq \mu_1$, and $x\in S_n$ is $321$-avoiding, there are ladder representations $\pi_1,\pi_2$ such that $\pi$ is the unique irreducible quotient of $\pi_1\times\pi_2$.
\end{corollary}

The above corollary served as an impetus to the definition of RSK-standard modules in \cite{gur-lap}. For a general irreducible $\pi= L(\gotM)$, ladder representations $\pi_1, \ldots,\pi_k$ can be constructed in a canonical manner which again utilizes the Robinson-Schensted-Knuth correspondence and produces $\pi$ as a unique irreducible quotient of $\pi_1\times\cdots\times\pi_k$.

\bibliographystyle{alpha}
\bibliography{propo2}{}

\end{document}